\documentclass[12pt]{amsart}
\usepackage{a4wide,graphicx}
\usepackage{color}
\usepackage{caption}
\usepackage{comment}
\usepackage{enumerate}
\usepackage{upgreek}
\usepackage[left=1in, right=1in]{geometry}
\usepackage{bbm,yfonts}
\usepackage{amsmath,hyperref}  
\usepackage{xcolor} 

\let\pa\partial   
  
\let\eps\varepsilon

\newcommand{\R}{{\mathbb R}}  

\newcommand{\diver}{\operatorname{div}}
\newcommand{\Rey}{\operatorname{Re}}

\newcommand{\D}{\operatorname{D}}
\newcommand{\T}{\mathsf{T}}

\newcommand{\dd}{{\mathrm{d}}}

\newcommand{\vphi}{\boldsymbol\varphi}

\newcommand{\uvro}{\upvarrho}
\newcommand{\bv}{\boldsymbol v}

\newcommand{\bu}{\boldsymbol u}
\newcommand{\bs}{\boldsymbol}
\newcommand{\bx}{\boldsymbol x}
\newcommand{\bn}{\mathbf n}

\newcommand{\bef}{\boldsymbol f}

\newcommand{\vol}{\operatorname{vol}}

\newcommand{\va}{\mathsf v}
\newcommand{\bva}{\boldsymbol{\mathsf v}}
\newcommand{\paa}{\mathsf p}

\theoremstyle{plain}
\newtheorem{theorem}{Theorem}[section]   
\newtheorem{lemma}[theorem]{Lemma}   
\newtheorem{proposition}[theorem]{Proposition}  
\newtheorem{corollary}[theorem]{Corollary}  

\theoremstyle{definition}
\newtheorem{definition}{Definition}[section]

\theoremstyle{remark}
\newtheorem{remark}{Remark}[section]

\begin{document}

\title[Nonlinear thin FSI problem]{Justification of a nonlinear sixth-order thin-film 
equation as the reduced model for a fluid - structure interaction problem}
 
\author{Mario Bukal$^1$}
\address[1]{University of Zagreb,
Faculty of Electrical Engineering and Computing\newline
Unska 3, 10000 Zagreb, Croatia}
\email{mario.bukal@fer.hr}
\author{Boris Muha$^2$}
\address[2]{University of Zagreb,
Faculty of Science, Department of Mathematics,
Bijeni\v cka cesta 30, 10000 Zagreb, Croatia}
\email{borism@math.hr}

\thanks{This work has been supported by the Croatian Science
Foundation under projects 7249 (MANDphy) and IP-2018-01-3706 (FSIApp).}



 \begin{abstract}
Starting from a nonlinear 2D/1D fluid-structure interaction problem between a thin layer
of a viscous fluid and a thin elastic structure, on the vanishing limit of the relative fluid thickness, we rigorously derive a sixth-order thin-film equation describing the dynamics of vertical displacements of the structure. The procedure is essentially based on quantitative energy estimates, quantified in terms of the relative fluid thickness, and a uniform no-contact result between the structure and the solid substrate. The sixth-order thin-film equation is justified in the sense of strong convergence of rescaled structure displacements to the unique positive classical solution of the thin-film equation. Moreover, the limit fluid velocity and the pressure can be expressed solely in terms of the solution to the thin-film equation.
\end{abstract} 

\date{\today}
\maketitle

\section{Introdcution}

Motivated by applications in microfluidics \cite{HoMa04, LBS05, OYN13, TaVe12}  
and so called lab-on-a-chip technologies \cite{DaFin06, SSA04}, 
which revolutionized experimentations in biochemistry and biomedicine, in this paper
we rigorously justify a nonlinear sixth-order thin-film equation as the reduced model
of a fluid-structure interaction (FSI) system in the lubrication approximation regime, 
i.e.~in the regime of the vanishing relative 
fluid thichness. The sixth-order thin-film 
equation, which is the subject of this paper, reads
\begin{equation}\label{eq:intro6TF}
\pa_th = \pa_x\left(h^3\left(\pa_x^5h - \Phi\right) \right),
\end{equation}
where $h(x,t) > 0$ denotes the rescaled fluid height and $\Phi$ is an external potential
resulting from the bulk fluid force.

Physical systems in which fluids lubricate underneath elastic 
structures are common also in other area of science and technology.
To name few like the growth of magma intrusions \cite{LPN13, Mic11}, 
subglacial floods \cite{DJBH08}, the passage of air flow in the lungs \cite{HHS08} 
and the operation of vocal cords \cite{Tit94}, or manufacturing of silicon 
wafers \cite{HuSo02, King89} and 
suppression of viscous fingering \cite{PIHJ12, PJH14}.
Contrary to the well-established FSI approach \cite{BGN14,BCMG16,Dow15},
favorable models for the above listed examples (and others) in engineering literature 
like \cite{HBB13,HoMa04,HuSo02, King89,LPN13,TaVe12} are sixth-order
evolution equations similar to (\ref{eq:intro6TF}). 
While an FSI problem is a coupled systems of partial differential 
equations on moving boundary domain, where the fluid is typically described by 
the Stokes or Navier-Stokes equations, and the structure is described by appropriate 
elasticity equations, the sixth-order thin-film equation is a 
single nonlinear evolution equation, 
which is typically easier to analyze and treat numerically.
In that sense it can be seen as a reduced (simplified) model of a given (multi-)physical 
system.

Formally, the thin-film equation results from an FSI problem by performing the 
so called lubrication approximation procedure \cite{Sze12}, 
giving rise to the Reynolds equation for the pressure (see e.g.~\cite{BayCha86,NazPil90}), 
and balancing the fluid pressure distribution with structure force densities per unit area. 
It is our aim here to perform this passage from an FSI problem to the thin-film equation 
rigorously in the sense of convergence of solutions of the FSI problem to the solution of the
corresponding thin-film equation, as stated in our main result in Theorem \ref{tm:main}.
Such procedure has been already performed by the authors 
in passing from fully linear 3D/2D \cite{BuMu21R} and 3D/3D \cite{BuMu21P} 
FSI problems to 2D linear sixth-order evolution equations. Now we extend these ideas from
linear to the nonlinear setting, but only in case of 2D/1D FSI problems due to 
availability of the global well-posedness result (cf.~Theorem \ref{ExistenceGlobal}).
However, this is the first rigorous justification of the nonlinear sixth-order thin-film 
equation (\ref{eq:intro6TF}) in the literature.

There are other similar models, which can be called reduced in this context, that 
have been subject of rigorous derivation, we outline them briefly in the sequel. 
Starting from various FSI problems, authors in 
\cite{CanMik03,MikGuCan07} studied the flow through a 
long elastic axially symmetric channel and using asymptotic expansion techniques obtained several 
reduced models of Biot-type. In \cite{CanMik03} they provided a rigorous 
justification of the reduced model through a weak convergence result and the corresponding 
error estimates. In \cite{PaSt06} Panasenko and Stavre analyzed a periodic flow in thin 
channel with visco-elastic walls. The problem was initially described by a 
linear 2D/1D FSI model, and under a special ratio of the channel 
height and the rigidity of the wall a linear sixth-order evolution equation 
emanated as the reduced model. 
A similar problem has been also considered in \cite{CuMP18}, 
resulting again in the reduced model described by another linear sixth-order equation. 
In both papers, reduced models have been rigorously justified by the appropriate convergence 
results. 
The starting point for 2D/1D FSI problem in \cite{PaSt06} has been justified in \cite{PaSt14} 
as a reduced model of a 2D/2D FSI problem. Finally, we mention \cite{PaSt20} where an interaction 
between a thin cylindrical elastic tube and a viscous fluid filling its thin interior 
has been considered. Using asymptotic analysis techniques, ten different cases have been identified
and related to the dependence of the Young modulus and density of the elastic medium with respect to small
geometric parameters, and corresponding reduced models have been derived. 
We emphasize that in all the above listed literature samples, FSI problems are linear.

On the other hand, widely known and well-studied models in mathematical literature 
are fourth-order thin-film equations
\begin{equation}\label{eq:intro4TF}
\pa_th = \pa_x\left(h^n\pa_x^3h\right),
\end{equation}
possibly in higher dimensions and with some lower-order terms that we for simplicity omitt \cite{ODB97}.
They can be seen as reduced models of free-bounary viscous fluid flows with dominant surface
tension effects \cite{Mye98}. Parameter $n>0$ reflects different physical settings, 
for $n=3$ equation (\ref{eq:intro4TF}) describes the dynamics of a liquid droplet on a solid 
substrate with no-slip condition, while $n=2$ corresponds to the Navier slip \cite{ODB97}. 
If $n=1$, then
(\ref{eq:intro4TF}) describes the dynamics of a thin fluid neck in the Hele-Shaw cell \cite{CDGKSZ93}. 
Some of these equations have been also subject of rigorous derivation.  
The lubrication approximation in the Hele-Shaw cell ($n=1$) has been rigorously 
justified in \cite{GiOt03,KnMa13,KnMa15,MaPr12}, while the sibling of our sixth-order equation,
equation (\ref{eq:intro4TF}) with $n=3$, has been justified in \cite{GuPr08}.
We emphasize at this point that the derivation in \cite{GuPr08} has been performed
under the assumption that the fluid droplet will not break up, while here we use
the no-contact result from \cite{grandmont2016existence} and prove its 
uniform strict positivity (cf.~Proposition \ref{prop:no_contact}), 
which is one of the key ingredients in
the justification of equation (\ref{eq:intro6TF}).

Let us now briefly describe our framework of the rigorous derivation of equation (\ref{eq:intro6TF}).
We start with an FSI problem for which we assume certain scalling assumptions for its non-dimensional 
coefficients. Then we derive a quantitative energy estimate, which is the source of weak convergence results.
Finally, based on the uniform no-contact result, we identify the sixth-order thin-film equation 
satisfied by the limit of rescaled structure displacements.

\subsection{FSI problem}  
The fluid domain at time $t$ is assumed to be a subgraph of
a space-time dependent function $\eta = \eta(x,t)$ describing the dynamics of 
the vertical structure displacement, i.e.~the fluid domain is of the form 
\begin{equation*}
\Omega_{\eta}(t)=\{\bx = (x,z):x\in \omega,\; z\in (0,\eta(x,t))\}\subset\R^2\,
\end{equation*}
with $\omega=(0,1)$. In applications $\Omega_\eta(t)$ can be seen as a periodic channel  
with deformable top wall (see Figure \ref{fig:domain_eta}). 
Let us further denote the space-time cylinder
$$
\Omega_\eta(t)\times (0,T):=\bigcup_{t\in 
(0,T)}\Omega_{\eta}(t)\times\{t\}\subset\R^2\times (0,\infty),\quad T\in (0,\infty]\,,
$$
to be domain of our free boundary problem.
The non-dimensionalized nonlinear FSI problem is described by the coupled system of partial differential equations:    
\begin{align}
\pa_t\bv+(\bv\cdot \nabla)\bv - \diver\sigma_f(\bv,p) &= \bs f\,,\quad 
\Omega_\eta(t)\times(0,\infty)\,,\label{1.eq:stokes}\\
\diver \bv &= 0\,,\quad \Omega_\eta(t)\times(0,\infty)\,,\label{1.eq:divfree}\\
\uvro\pa_{tt}\eta - \updelta \partial^2_{x}\pa_t\eta + \upbeta\partial^4_{x}\eta  
&=-J^{\eta}(x,t) \big(\sigma_f(\bv,p){\bf n}^{\eta}\big )(x,\eta(x,t),t)\cdot{\bf e}_z\,, 
\quad \omega\times(0,\infty)\,,\label{1.eq:elast}
\\
\bv(x,\eta(x,t),t)&=(0,\partial_t\eta(x,t))\,,\quad \omega\times(0,\infty)\,.\label{KinematicBC}
\end{align}
Equations \eqref{1.eq:stokes} and \eqref{1.eq:divfree} are incompressible 
Navier-Stokes equations describing the flow of a viscous Newtonian fluid of velocity $\bv$ and
pressure $p$. The Cauchy stress tensor is given by $\sigma_f(\bv,p) = 2 \D(\bv) - pI_2$, where
$\D(\bv) = \frac{1}{2}(\nabla\bv + (\nabla\bv)^T)$ denotes the symmetric part of the gradient, 
and $\bs f$ denotes the density of the fluid external force.
The structure is described by a linear equation of visco-elastic plate (\ref{1.eq:elast}), where 
$J^{\eta}(x,t)=\sqrt{1+\partial_x\eta(x,t)^2}$ is the Jacobian of the transformation 
from Eulerian to Lagrangian coordinates, $\bf n^{\eta}$ is the unit outer normal to the deformed configuration 
$\Omega_{\eta}$, ${\bf e}_z$ is the unit vector in $z$ direction, and 
$\uvro$, $\updelta$, $\upbeta$ are positive parameters describing the density, 
visco-elasticity and bending properties of the structure, respectively.
Equations for the fluid and the structure are coupled via dynamic and kinematic 
conditions \eqref{1.eq:elast} and \eqref{KinematicBC} representing the balance of forces 
in ${\bf e}_z$ direction and continuity of the velocity, respectively.
For dimensional version of the problem \eqref{1.eq:stokes}--\eqref{KinematicBC} see
Appendix \ref{app:SA}.

\begin{figure}
\includegraphics[width = 9cm]{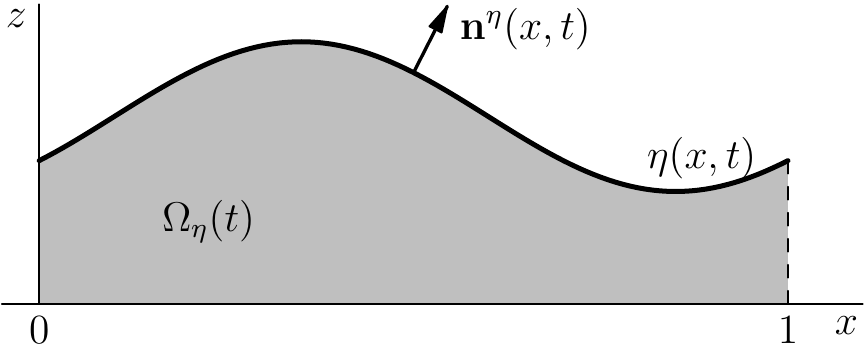}  
\caption{Sketch of the time-dependent domain $\Omega_{\eta}(t)$.}
\label{fig:domain_eta}     
\end{figure}

Assumption that the structure moves only in the vertical direction is a simplification which
is not fully justified from the physical grounds. However, it is a reasonable assumption from
the point of view of our previous results on the dimension reduction in thin FSI problems \cite{BuMu21P},
where it has been shown that the vertical displacement dominates horizontal displacements by the order of magnitude.
For more details about the physical background of system \eqref{1.eq:stokes}-\eqref{KinematicBC} and 
corresponding lower-dimensional elasticity models we refer to \cite{Cia97,BorSun} and 
the references therein. 
The bottom boundary is a rigid substrate and we prescribe the standard no-slip boundary condition 
for the fluid velocity:
$\bv(x,0,t)=0$ for all $(x,t)\in\omega\times(0,\infty)$.
On the lateral boundaries of $\Omega_\eta(t)$ we prescribe periodic boundary conditions 
in the horizontal direction,
which is taken for technical simplicity and because of availability of the 
global existence results \cite{grandmont2016existence}. 
In such a case the flow is driven by the right-hand side $\bs f$.
Finally, for simplicity of exposition, we impose trivial initial conditions for velocities
$\bv(\cdot,0)=0$ and $\partial_t\eta(\cdot,0)=0$, while for the displacement we take 
$\eta(\cdot,0)=\eta_0$ and assume $\eta_0(x) > 0$ for all $x\in\omega$.
The incompressibility condition (\ref{1.eq:divfree}) and the kinematic coupling (\ref{KinematicBC}) 
provide the conservation of the fluid volume:
\begin{align}
\frac{\dd}{\dd t}{\rm vol}(\Omega_{\eta}(t)) &= \frac{\dd}{\dd t}\int_{\omega}\eta(x,t)\dd x\nonumber
= \int_{\omega}\partial_t\eta(x,t)\dd x = \int_{\omega}\bv(x,\eta(x,t),t)\cdot{\bf e}_z\dd x\\
&=\int_{\omega} (\bv\cdot \bn^\eta)(x,\eta(x,t),t)\bn^\eta(x,\eta(x,t),t)\cdot{\bf e}_z\dd x \label{VolumeCon}
= \int_{\pa\Omega_{\eta}(t)}\bv\cdot\bn^\eta\, \dd S\\
&= \int_{\Omega_\eta(t)}\diver\bv\, \dd \bx = 0\,.\nonumber
\end{align}
Therefore, despite the periodic boundary conditions, the displacement $\eta$ is uniquely determined.

\subsubsection*{Scaling assumptions}
Guided by our previous results in the linear case \cite{BuMu21P}, we assume the following scaling 
ansatz which relates the nondimensional structure parameters and the time scale $\T$ to 
the order of the relative fluid thickness $\eps$:
\begin{enumerate}
  \item[(S1)] $\upbeta = \hat{\upbeta}\eps^{-1}$, $\updelta=\hat{\updelta}\eps^{-r}$, for some $r\in [1,3]$, 
  	and $\uvro = \hat{\uvro}\eps$, where $\hat{\upbeta},\;\hat{\updelta},\;\hat{\uvro}>0$ are independent of $\eps$;
  \item[(S2)] $\T = \eps^{-2}$.
\end{enumerate}
Scaling assumptions (S1) are motivated by applications in microfluidics 
and discussed in detail in Appendix \ref{app:SA}, while (S2)
is the standard time scale in the lubrication approximation regime.
In this context we also assume that the volume of the initial domain $\Omega_\eta(0)$ is of 
size $O(\eps)$. More precisely, 
\begin{enumerate}
  \item[(S3)] $\eta^{\eps}_0(x)=\eps\eta_0(x)$ for some $\eta_0(x)$ independent of $\eps$. 
\end{enumerate}
Finally, we assume that the fluid volume force satisfies 
\begin{enumerate}
  \item[(S4)] $\|\bs f\|_{L^\infty(0,\infty;L^{\infty}(\Omega_\eta(t);\R^2))} \leq C$,
\end{enumerate}
where $C>0$ is independent of $\eps$.
Assumption (S4) is verified by many physically relevant forces. 

\subsection{Uniform estimates}
Let $\eta_0\in H_\#^3(\omega)$ be given strictly positive $\omega$-periodic function, then according to \cite{grandmont2016existence}
(cf.~Theorem \ref{ExistenceGlobal} below) for every $\eps > 0$ there exists a unique 
global-in-time strong solution $(\bv^\eps,p^\eps,\eta^\eps)$ to the system 
(\ref{1.eq:stokes})-(\ref{KinematicBC}) with initial displacement $\eta^{\eps}_0=\eps\eta_0$.
We use this stronger solution concept (cf.~Definition \ref{def:strongsol}) because of availability of the global well-posedness result and moreover, the strict positivity of $\eta^\eps$ meaning that the contact between the elastic structure and the rigid substrate will not occur in finite time. On the other hand, global well-posedness of weak solutions (see Definition \ref{def:weaksol}) to (\ref{1.eq:stokes})-(\ref{KinematicBC}) is an open problem. Namely, one can prove their existence only
up to the contact between the elastic structure and the rigid substrate \cite{BorSun}, which is unresolved issue in the context of weak solutions.

Under scalling assumptions (S1)-(S4), strong solutions $\bv^\eps$ and $\eta^\eps$ satisfy the following energy estimate (cf.~Proposition \ref{prop:energyineq} below): for a.e.~$t\in(0, T)$ it holds
\begin{align*}
\frac{1}{2}\|\bv^\eps(t)\|^2_{L^2(\Omega_{\eta}(t))}\nonumber
&+\frac{\eps^{-2}}{2}\int_0^t\!\!\int_{\Omega_{\eta}(s)}|\nabla\bv^\eps|^2\dd \bx\dd s\\
+\frac{\uvro\eps^{5}}{2}\|\partial_t\eta^\eps(t)\|^2_{L^2(\omega)}
&+\updelta\eps^{2-r}\int_0^t\|\partial_t\partial_{x}\eta^\eps(s)\|^2_{L^2(\omega)}\dd s
+\frac{\upbeta \eps^{-1}}{2}\|\partial^2_x\eta^\eps(t)\|^2_{L^2(\omega)}
\leq C\eps\,,
\end{align*}
where $C>0$ is independent of $\eps$, and $T>0$ is a rescaled time horizon. This inequality is the key source of a priori estimates on strong solutions and consequently weak convergence results.
Another indispensable ingredient is the uniform no-contact result (see Proposition \ref{prop:no_contact}), i.e.~there exists a constant $c>0$, independent of $\eps$, such that 
\begin{equation*}
\frac{\eta^\eps(x,t)}{\eps} \geq c\quad\text{for a.e.~}(x,t)\in \omega\times(0,T)\,.
\end{equation*}   

\subsection{Reduced model}
Having these at hand we can prove
\begin{equation*}
\frac{\eta^{\eps}}{\eps}\to h\quad \text{strongly in }C([0,T];C^1_\#(\overline\omega))\ \text{ as }\eps\downarrow0\,,
\end{equation*}
and function $h$ can be identified as the unique positive classical solution of the following nonlinear sixth-order thin-film equation
\begin{align}\label{1.eq:6tf}
\pa_t h = \pa_x\left(h^3\left(\frac{\upbeta}{12}\partial^5_{x}h 
- \chi_{\{r=3\}}\frac{\updelta}{12}\partial^3_{x}\partial_{t}h - \Phi \right) \right)\,\,\quad \text{on }\omega\times(0,T)\,
\end{align}
with initial datum $\eta_0$.
Here $\Phi(x,t)$ is an external potential related to the fluid volume force $\bs f$ as given by (\ref{limitvelint}), and
$\chi_{\{r=3\}} = 1$ if $r=3$ and $0$, otherwise.  
In this sense, equation (\ref{1.eq:6tf}) can be understood as the reduced model for the FSI problem
(\ref{1.eq:stokes})-(\ref{KinematicBC}) in the limit
as $\eps\downarrow0$. 
\begin{remark}
Note that equation (\ref{eq:intro6TF}) is just a special case of equation \eqref{1.eq:6tf} for 
$\upbeta = 12$ and $\chi_{\{r=3\}} = 0$, i.e.~$r<3$ in (S1).
\end{remark}

We summarize our main findings in
\begin{theorem}\label{tm:main}
Let $(\bv^{\eps},p^\eps, \eta^{\eps})$ be a family of strong solutions to 
problem \eqref{1.eq:stokes}-\eqref{KinematicBC} in the sense of Definition \ref{def:strongsol} and assume that coefficients and data satisfy (S1)-(S4). 
Then on the limit as $\eps\downarrow0$ 
\begin{align}
\eps^{-1}\eta^{\eps}\to h\quad \text{strongly in}\; C([0,T];C^1_\#(\overline \omega))\,,
\end{align}
where $h$ is the unique positive classical solution of equation \eqref{1.eq:6tf}.
Moreover, rescaled strong solutions $\hat\bv^\eps$ and $\hat p^\eps$ defined on the reference domain $\Omega\times (0,T)$, where $\Omega = \omega\times(0,1)$, satisfy:
\begin{align}
\eps^{-2}\hat\bv^{\eps}\rightharpoonup (v_1 , 0)\; \text{weakly in}\; L^2(0,T;L^2(\Omega))\,,\\
\hat p^{\eps} \rightharpoonup p \quad \text {weakly in } H^{-1}(0,T;L^2(\Omega))\,,
\end{align}
where $p$ and $v_1$ are given by
\begin{align*}
p &= \upbeta\partial^4_{x}h - \chi_{\{r=3\}}\updelta\partial_{t}\partial^2_{x}h\,,\\
v_1(\cdot,y,\cdot)&=\frac{1}{2}y(y - 1)h^2\partial_x p + h^2F(\cdot,y,\cdot)\,,\quad y\in [0,1]\,,
\end{align*}
and $F$ is given as in \eqref{LimitVelFormula}.
\end{theorem}

In Section \ref{sec:EE} we derive the basic energy estimate for classical solutions of FSI problem 
\eqref{1.eq:stokes}-\eqref{KinematicBC} and discuss different solution concepts. Quantitative uniform estimates with respect to small parameter $\eps$ have been conducted in Section \ref{sec:UE}, while the identification of the reduced model model, i.e.~the proof of Theorem \ref{tm:main} has been performed in Section \ref{sec:RM}. We conclude this paper with brief Section \ref{sec:CP} on future perspectives and two appendices discussing the physical background of FSI problem \eqref{1.eq:stokes}-\eqref{KinematicBC} and proving a technical lemma, respectively.

\section{Energy estimates and global solutions}\label{sec:EE}
In this section we provide quantitative bounds on the energy and the
energy dissipation of the system \eqref{1.eq:stokes}-\eqref{KinematicBC},
which depend explicitly on the small parameter $\eps$. 

\subsection{Auxiliary inequalities}
First we provide basic functional inequalities tailored to our moving boundary domains
$\Omega_\eta(t)\subset\R^2$ and solutions of the system \eqref{1.eq:stokes}-\eqref{KinematicBC}.
\begin{proposition}\label{Poincare}
Let $T>0$, $\eta\in W^{1,\infty}(0,T;X)$, where $X$ is a Banach space and let
$\bv\in\left\{\bv\in H^1(\Omega_{\eta}(t))\ :\ \diver\bv=0,\;
 \bv|_{z=0}=0,\ \bv \text{ is } \omega\text{-periodic in }x\right\}$ 
for $t\in(0,T)$. The following (in)equalities hold for a.e.~$t\in(0,T)$:
\begin{align}
& \|\bv\|_{L^2(\Omega_\eta(t))} \leq 
\frac{1}{\sqrt{2}}\|\eta(t)\|_{L^{\infty}(\omega)}\|\partial_z\bv\|_{L^2(\Omega_{\eta}(t))}\,, 
\quad \text{(Poincar\'e inequality)}\,,
\label{2.ineq:eps_P}\\
&  \sqrt{2}\|\D(\bv)\|_{L^2(\Omega_\eta(t))}
= \|\nabla\bv\|_{L^2(\Omega_\eta(t))} \,,\quad \text{(Korn equality)}\,.\label{2.ineq:eps_K} 
\end{align}
\end{proposition}
\begin{proof}
Utilizing the Cauchy-Schwarz inequality and the no-slip boundary condition
at $z=0$, we calculate: for a.e.~$t\in(0,T)$
\begin{align*}
\|\bv\|_{L^2(\Omega_\eta(t))}^2 &= \int_{\omega}
\int_{0}^{\eta(x,t)}\bv(x,z)^2\dd \bx
=\int_{\omega}\int_{0}^{\eta(x,t)}\left (\int_{0}^{z}\partial_z\bv(x,\zeta)\dd \zeta\right )^2\dd \bx
\\
&\leq \int_{\omega}\int_{0}^{\eta(x,t)}\left( z\int_{0}^{z}(\partial_z\bv)^2(x,\zeta)\dd \zeta\right)\dd \bx\\
&\leq \int_{\omega}\int_{0}^{\|\eta(t)\|_{L^{\infty}}}\left(z\int_{0}^{\eta(x,t)}(\partial_z\bv)^2(x,\zeta)\dd \zeta\right)\dd\bx
=\frac{1}{2} \|\eta(t)\|_{L^{\infty}(\omega)}^2 \|\partial_z\bv\|_{L^2(\Omega_\eta(t))}^2\,,
\end{align*}
which implies the Poincar\'e inequality (\ref{2.ineq:eps_P}). 

The Korn equality (\ref{2.ineq:eps_K}) follows directly from the fact that the structure displacement 
is only vertical and that the fluid velocity is divergence free (cf.~\cite[Lemma 6]{CDEM} 
or \cite[Lemma A.5]{LenRuz}).
\end{proof}

\subsection{Energy estimates}
Testing formally equations (\ref{1.eq:stokes}) and (\ref{1.eq:elast}) with assumed 
classical solutions $\bv$ and $\pa_t\eta$,
respectively, and integrating by parts yields the basic energy inequality:
for every $t\in (0,T)$
\begin{align}
\frac{1}{2}\|\bv(t)\|^2_{L^2(\Omega_{\eta}(t))}\nonumber
+2\int_0^t\!\!\int_{\Omega_{\eta}(s)}|\D(\nabla\bv)|^2\dd \bx\dd s\\
+\frac{\uvro}{2}\|\partial_t\eta(t)\|^2_{L^2(\omega)} \label{EI}
+\updelta\int_0^t\|\partial_t\partial_{x}\eta(s)\|^2_{L^2(\omega)}\dd s
+\frac{\upbeta}{2}\|\partial^2_x\eta(t)\|^2_{L^2(\omega)}\\
\leq \frac{\upbeta}{2}\|\partial_x^2 \eta_0\|^2_{L^2(\omega)}\nonumber
+\int_0^t\!\!\int_{\Omega_{\eta}(s)}\bs f\cdot\bv\,\dd\bx\dd s\,. 
\end{align}
Let us now estimate the force term. Employing the Cauchy-Schwarz, the Poincar\'e inequality
(\ref{2.ineq:eps_P}), assumption (S4) and conservation of the volume, we estimate:
\begin{align*}
\left|\int_0^t\int_{\Omega_\eta(s)}\bs f\cdot\bv\,\dd\bx\dd s\right|
&\leq \int_0^t\|\bs f\|_{L^2(\Omega_{\eta}(s))}\|\bv\|_{L^2(\Omega_{\eta}(s))}\dd s\\
&\leq C\left(\vol(\Omega_\eta(t))\right)^{1/2}
\int_0^t\|\eta(s)\|_{L^{\infty}(\omega)}\|\partial_z\bv\|_{L^2(\Omega_{\eta}(s))}\dd s\\
&\leq C\vol(\Omega_{\eta_0})\int_0^t\|\eta(s)\|_{L^{\infty}(\omega)}^2\dd s +
 \int_0^t\|\D(\bv)\|_{L^{2}(\Omega_{\eta}(s))}^2\dd s\,,
\end{align*}
where $C>0$ is a positive constant independent of all variables.

Therefore, the energy inequality (\ref{EI}) can be closed in the following form: for every 
$t\in(0,T)$
\begin{align}
\frac{1}{2}\|\bv(t)\|^2_{L^2(\Omega_{\eta}(t))}\nonumber
+\int_0^t\!\!\int_{\Omega_{\eta}(s)}|\D(\bv)|^2\dd \bx\dd s\\
+\frac{\uvro}{2}\|\partial_t\eta(t)\|^2_{L^2(\omega)} \label{EIEps}
+\updelta\int_0^t\|\partial_t\partial_{x}\eta(s)\|^2_{L^2(\omega)}\dd s
+\frac{\upbeta}{2}\|\partial^2_x\eta(t)\|^2_{L^2(\omega)}\\
\leq \frac{\upbeta}{2}\|\partial_x^2 \eta_0\|^2_{L^2(\omega)}
+ C\vol(\Omega_{\eta_0})\int_0^t\|\eta(s)\|_{L^{\infty}(\omega)}^2\dd s\,.
\nonumber
\end{align}

Next, we estimate the second term on the right-hand side. 
Using the continuity of the Sobolev embedding $H^1(\omega)\hookrightarrow L^\infty(\omega)$, 
there exists a constant $C_S>0$ such that for every $t\in(0,T)$ we have
$\|\eta(t)\|_{L^{\infty}(\omega)} \leq C_S\|\eta(t)\|_{H^1(\omega)}$.
Since the volume of $\Omega_\eta(t)$ is preserved, 
i.e.~$\int_\omega \eta(t) \dd x = \int_\omega \eta_0 \dd x$, the Poincar\'e inequality provides
\begin{equation*}
\|\eta(t) - \overline{\eta}_0\|_{L^2(\omega)} \leq C_P\|\partial_x\eta(t)\|_{L^2(\omega)}\,,
\end{equation*}
where $\overline\eta_0 = L^{-1}\int_\omega \eta_0 \dd x$ and $C_P > 0$ is the Poincar\'e constant.
The triangle inequality then gives
\begin{equation*}
\|\eta(t)\|_{L^2(\omega)} \leq \|\eta(t) - \overline{\eta}_0\|_{L^2(\omega)} + \|\overline{\eta}_0\|_{L^2(\omega)} 
\leq C\left(\|\partial_x\eta(t)\|_{L^2(\omega)} + \overline{\eta}_0\right)\,,
\end{equation*}
and another application of the Poincar\'e inequality for $\omega$-periodic functions yields
\begin{equation*}
\|\eta(t)\|_{H^1(\omega)} \leq C\left(\|\partial_{x}^2\eta(t)\|_{L^2(\omega)} + \overline{\eta}_0\right)\,.
\end{equation*}
Therefore, the right-hand side of \eqref{EIEps} can be controlled with
\begin{equation*}
\frac{\upbeta}{2}\|\partial_x^2 \eta_0\|^2_{L^2(\omega)}
+ C\vol(\Omega_{\eta_0})\int_0^t\|\pa_x^2\eta(s)\|_{L^{2}(\omega)}^2\dd s + 
Ct\vol(\Omega_{\eta_0})\overline{\eta}_0^2\,.
\end{equation*}
In order to close the energy estimate, we finally employ the Gr\"onwall inequality.
Namely,
\begin{align*}
\frac{\upbeta}{2}\|\partial^2_x\eta(t)\|^2_{L^2(\omega)} \leq 
\frac{\upbeta}{2}\|\partial_x^2 \eta_0\|^2_{L^2(\omega)} + 
Ct\vol(\Omega_{\eta_0})\overline{\eta}_0^2
+ C\vol(\Omega_{\eta_0})\int_0^t\|\pa_x^2\eta(s)\|_{L^{2}(\omega)}^2\dd s
\end{align*}
implies
\begin{equation*}
\|\partial^2_x\eta(t)\|^2_{L^2(\omega)} \leq 
\left(\|\partial_x^2 \eta_0\|^2_{L^2(\omega)} + 
2Ct\vol(\Omega_{\eta_0})\overline{\eta}_0^2/\upbeta \right)
\exp\left(Ct\vol(\Omega_{\eta_0})/\upbeta\right)\,.
\end{equation*}
Taking into account the scalings (S1)-(S2) and smallness of the initial data (S3) we find
\begin{equation}\label{2.ineq:paxxeta}
\|\partial^2_x\eta(\hat t)\|^2_{L^2(\omega)} \leq C\eps^2(1 + \hat t)\exp(C\hat t) \leq C\eps^2
\end{equation}
for all $\hat t\in(0,\hat T)$, where $\hat t = t/\T$ and $\hat T$ is the rescaled time horizon.
Employing (\ref{2.ineq:paxxeta}) and the Korn equality 
(\ref{2.ineq:eps_K}) in (\ref{EIEps}), together with the scalings (S1)-(S2) and smallness of the initial data, 
we arrive to the energy estimate: for every $\hat t\in(0,\hat T)$
\begin{align}
\frac{1}{2}\|\bv(\hat t)\|^2_{L^2(\Omega_{\eta}(\hat t))}\nonumber
+\frac{\eps^{-2}}{2}\int_0^{\hat t}\!\!\int_{\Omega_{\eta}(s)}|\nabla\bv|^2\dd \bx\dd s\\
+\frac{\hat\uvro\eps^{5}}{2}\|\partial_{\hat t}\eta(\hat t)\|^2_{L^2(\omega)} \label{EIEps3}
+\hat\updelta\eps^{2-r}\int_0^{\hat t}\|\partial_{\hat t}\partial_{x}\eta(s)\|^2_{L^2(\omega)}\dd s
+\frac{\hat\upbeta \eps^{-1}}{2}\|\partial^2_x\eta(t)\|^2_{L^2(\omega)} & \leq C\eps\,.
\end{align}
Neglecting the hats in the sequel, we have proved the following key energy estimate.
\begin{proposition}\label{prop:energyineq}
Let $\bv$ and $\eta$ be classical solutions of the system \eqref{1.eq:stokes}-\eqref{KinematicBC}
under scaling assumptions (S1)-(S4), then for every $ t\in(0, T)$ it holds
\begin{align}
\frac{1}{2}\|\bv(t)\|^2_{L^2(\Omega_{\eta}(t))}\nonumber
&+\frac{\eps^{-2}}{2}\int_0^t\!\!\int_{\Omega_{\eta}(s)}|\nabla\bv|^2\dd \bx\dd s\\
+\frac{\uvro\eps^{5}}{2}\|\partial_t\eta(t)\|^2_{L^2(\omega)} \label{EIEkey}
&+\updelta\eps^{2-r}\int_0^t\|\partial_t\partial_{x}\eta(s)\|^2_{L^2(\omega)}\dd s
+\frac{\upbeta \eps^{-1}}{2}\|\partial^2_x\eta(t)\|^2_{L^2(\omega)}
\leq C\eps\,,
\end{align}
where $C>0$ is independent of $\eps$ and all variables.
\end{proposition} 

\subsection{Weak and strong solutions}
Neglecting hats in the notation, in further we work with FSI problem \eqref{1.eq:stokes}-\eqref{KinematicBC}
under scaling assumptions (S1)-(S4), i.e.~the system is considered in the rescaled time.
Let us first introduce appropriate solution spaces.
The fluid solution space will depend on the displacement $\eta$. If we denote
\begin{equation*}
{V}_F(t)=\left\{\bv\in H^1(\Omega_{\eta}(t))\ :\ \diver\bv=0,\;
 \bv|_{z=0}=0,\ \bv \text{ is } \omega-\text{periodic in }x\right\},
\end{equation*}
then the above energy estimate suggests that, 
for a given time horizon $T>0$ appropriate fluid solution space is 
\begin{equation*}
{\mathcal V}_F(0,T;\Omega_{\eta}(t))
= L^{\infty}(0,T;L^2(\Omega_{\eta}(t)))\cap L^2(0,T;{V}_F(t))\,,
\end{equation*}
while for the structure, again based on the energy estimate, we choose the solution space to be 
\begin{equation*}
{\mathcal V}_S(0,T;\omega)
= W^{1,\infty}(0,T;L^2(\omega))\cap 
L^{\infty}(0,T;H^2_{\#}({\omega}))\cap H^1(0,T;H^1_\#(\omega))\,.
\end{equation*}

Employing the Reynolds transport theorem we find
\begin{align*}
\int_0^T\!\!\int_{\Omega_\eta(t)}\pa_t\bv\cdot\vphi\,\dd\bx\dd t =
- \int_0^T\!\!\int_{\Omega_\eta(t)}\bv\cdot\partial_t\vphi\,\dd\bx\dd t
- \eps^2\int_0^T\!\!\int_{\omega}(\pa_t\eta)^2\varphi_2\,\dd x\dd t
\end{align*}
for all test functions $\vphi\in C^1_c([0,T);\mathcal{V}_F(t))$. On the other hand,
performing integration by parts in the convective term we have
\begin{align*}
\eps^{-2}\int_0^T\!\!\int_{\Omega_\eta(t)}(\bv\cdot\nabla)\bv\cdot\vphi\,\dd \bx\dd t = 
- \eps^{-2}\int_0^T\!\!\int_{\Omega_\eta(t)}(\bv\cdot\nabla)\vphi\cdot\bv\,\dd \bx\dd t
+ \eps^{2}\int_0^T\!\!\int_{\omega}(\pa_t\eta)^2\varphi_2\,\dd x\dd t\,.
\end{align*}
In both inequalities we used kinematic boundary condition (\ref{KinematicBC}), which after 
rescaling of the time reads $\bv=(0,\eps^2\partial_t\eta)$ on $\omega$.
Therefore, summing up the last two identities we obtain the weak formulation of the inertial term
as in the following definition.

\begin{definition}\label{def:weaksol}
We call $(\bv,\eta)\in{\mathcal V}_F(0,T;\Omega_{\eta}(t))\times {\mathcal V}_S(0,T;\omega)$ a 
{\em weak solution} of the FSI problem (\ref{1.eq:stokes})-(\ref{KinematicBC}) 
if for every $(\vphi,\psi)\in C^1_c([0,T);\mathcal{V}_F(t)\times H^2_{\#}({\omega}))$ 
satisfying $\vphi(x,\eta(x,t),t)=\psi(x,t){\bf e}_z$ it holds
\begin{align}
-\int_0^T\!\!\int_{\Omega_\eta(t)}\left(\bv\cdot\partial_t\vphi
+\eps^{-2}(\bv\cdot\nabla)\vphi\cdot\bv\right)\dd \bx\dd t
+2\eps^{-2}\int_0^T\!\!\int_{\Omega_\eta(t)}\D(\bv):\nabla\vphi\,\dd\bx\dd t
\nonumber
\\
-\uvro\eps^{3}\int_0^T\!\!\int_{\omega}\partial_t\eta\partial_t\psi\,\dd x\dd t
+\updelta\eps^{-r}\int_0^T\!\!\int_{\omega}\partial_{x}\partial_t\eta\pa_x\psi\, \dd x\dd t
+\upbeta\eps^{-3}\int_0^T\!\!\int_{\omega}\partial^2_{x}\eta\partial^2_{x}\psi\, \dd x\dd t
\label{WeakFormulation}
\\
= \eps^{-2}\int_0^{T}\!\!\int_{\Omega_\eta(t)}\bs f\cdot\vphi\,\dd\bx\dd t
\nonumber
\end{align}
and the time rescaled version of \eqref{KinematicBC} is satisfied in the sense of traces. Moreover, the energy inequality 
(\ref{EIEkey}) is satisfied for a.e.~$t\in(0,T)$.
\end{definition}
The existence of weak solutions is by now well-established in the literature, see e.g. \cite{CDEM,BorSun}. 
However, the question of time globality of weak solutions is rather open. More precisely, the 
existence results assert the following: either $T=\infty$ or 
$\displaystyle\lim_{t\uparrow T}\min_{x\in\overline\omega}\eta(x,t) = 0$, i.e.~weak solutions exist 
as long as there is no contact between the elastic and rigid boundary.
Even though there are results that contact will not occur in the case when the structure is 
rigid \cite{HilTak}, to the best of our 
knowledge there are no global in time existence results for weak solutions to problem 
\eqref{1.eq:stokes}-\eqref{KinematicBC}. 

\begin{remark}
Unlike in the case of the standard Navier-Stokes equations with rigid walls, 
the pressure in problem \eqref{1.eq:stokes}-\eqref{KinematicBC} uniquely determined. 
Physically the reason is that the structure is deformable so it "can feel" the pressure. 
\end{remark}

Since the contact issue is unresolved for weak solutions, our subsequent analysis relies on the
concept of strong solutions for which the global well-posedness is available.
We start with the definition taken from \cite[cf.~Definition 1]{grandmont2016existence}.

\begin{definition}\label{def:strongsol}
We call $(\bv,p,\eta)$ a 
{\em strong solution} of the FSI problem  (\ref{1.eq:stokes})-(\ref{KinematicBC}) on $(0,T)$
if
\begin{align*}
\eta &\in H^2(0,T;L^2(\omega))\cap L^2(0,T;H^4_\#(\omega))\,,\quad \eta^{-1}\in L^\infty(0,T;L^\infty(\omega))\,,\\
\bv &\in H^1(0,T;L^2(\Omega_\eta(t)))\cap L^2(0,T;H^2(\Omega_\eta(t)))\,,\\
p&\in L^2(0,T;H^1(\Omega_\eta(t)))\,
\end{align*}
and equations \eqref{1.eq:stokes}--\eqref{1.eq:divfree} are satisfied a.e.~in $\Omega_\eta(t)\times(0,T)$, equations \eqref{1.eq:elast}--\eqref{KinematicBC}
are satisfied a.e.~in $\omega\times(0,T)$, and initial conditions and the no-slip boundary condition are satisfied a.e.
\end{definition}

The global well-posedness result is asserted by the following 
theorem also taken from \cite[cf.~Theorem 1]{grandmont2016existence}.
\begin{theorem}\label{ExistenceGlobal}
Let $\eta_0\in H_\#^3(\omega)$ be strictly positive and $T>0$ given time horizon. For every $\eps>0$ there exists a 
unique global-in-time strong solution $(\bv^\eps,p^\eps,\eta^\eps)$ to problem 
\eqref{1.eq:stokes}-\eqref{KinematicBC} with initial conditions 
$\bv_0=0$, $\eta_0^\eps = \eps{\eta}_0$, $\eta_1=0$ and the right hand 
side $\bs f\in L^\infty(0,T;L^\infty(\R^2))$. 
\end{theorem}

\begin{remark}
In fact the strong solutions in \cite{grandmont2016existence} have been constructed for FSI problem without the fluid volume force. 
But the result can be straightforwardly extended to include the right hand side $\bef$ in \eqref{1.eq:stokes}.
\end{remark}

\section{Uniform estimates}\label{sec:UE}

\subsection{Uniform estimates for the structure displacements}
Taking into account the volume preservation, the energy estimate (\ref{EIEkey}) immediately
gives
\begin{corollary} Let $(\eta^\eps)$ be a family of structure displacements constructed 
in Theorem \ref{ExistenceGlobal}. There exists a constant $C>0$, independent of $\eps$, such that 
\begin{equation}\label{DisplacementEstimate}
\|\eta^\eps\|_{L^{\infty}(0,T;H^2_\#(\omega))}\leq C\eps\,.
\end{equation}
\end{corollary}
\begin{proof}
Employing the Poincar\'e inequality twice, we find
\begin{align*}
&\|\eta^\eps\|_{L^{\infty}(0,T;L^2(\omega))}^2 + \|\pa_x\eta^\eps\|_{L^{\infty}(0,T;L^2(\omega))}^2 + 
\|\pa_x^2\eta^\eps\|_{L^{\infty}(0,T;L^2(\omega))}^2\\ &\quad \leq
C\|\pa_x\eta^\eps\|_{L^{\infty}(0,T;L^2(\omega))}^2 + C\overline\eta_0^2 + 
\|\pa_x^2\eta^\eps\|_{L^{\infty}(0,T;L^2(\omega))}^2\\
&\quad \leq C\left(\|\pa_x^2\eta^\eps\|_{L^{\infty}(0,T;L^2(\omega))}^2 + \overline\eta_0^2\right)\,.
\end{align*}
Energy inequality (\ref{EIEkey}) and assumption on smallness of the initial data now yield the
statement.
\end{proof}

Theorem \ref{ExistenceGlobal} states that for every $\eps>0$ there exists a
unique solution $(\bv^{\eps},\eta^{\eps})$ such that $\eta^{\eps}(x,t)>0$ for 
a.e.~$(x,t)\in \omega\times(0,\infty)$. 
On the other hand, from the previous corollary we conclude that $\eta^{\eps}(x,t)\to 0$ 
as $\eps\to 0$. Since our goal is to derive the effective equation for the first approximation 
of $\eta^{\eps}$ which is of the form $\eps h(x,t)$, our first step is to prove that $h(x,t)$ 
is strictly positive and uniformly bounded from below with a positive constant. This is precisely
the statement of the following proposition.

\begin{proposition}\label{prop:no_contact}
Let $(\eta^\eps)$ be a family of structure displacements constructed in Theorem \ref{ExistenceGlobal}.
There exists a constant $c>0$, independent of $\eps$, such that 
\begin{equation}\label{3.est:no_contact}
\frac{\eta^\eps(x,t)}{\eps} \geq c\quad\text{for a.e.~}(x,t)\in \omega\times(0,T)\,.
\end{equation}   
\end{proposition}

\noindent In the proof, we follow the arguments from \cite{grandmont2016existence} and adapt it to our setting. 
More precisely, we reprove \cite[Proposition 3]{grandmont2016existence}, but taking into 
account the scaling assumptions on the coefficients, initial data, and the energy 
estimate \eqref{EIEkey}. 
We will not repeat every detail here, but we focus on the estimates involving the small 
parameter $\eps$. The main idea is to test equation \eqref{1.eq:elast} 
with $\partial_x^2\eta^{\eps}$. Therefore, we need to construct the corresponding 
divergence-free test function for the fluid equation \eqref{1.eq:stokes}, which 
satisfies the kinematic coupling condition. This is achieved by defining the stream function
\begin{align}\label{def:stream_f}
\psi^{\eps}(x,z,t)=\partial_x\eta^{\eps}(x,t)\chi\left(\frac{z}{\eta^{\eps}(x,t)}\right),  
\end{align}
where $\chi(z)=z^2(3-2z)$ is a cut-off function. Necessary estimates for the stream function are 
collected in the following lemma, which is a version of \cite[Propostion 8]{grandmont2016existence} 
adapted to our setting.
\begin{lemma}[Stream function estimates]\label{StreamEstimates}
Let $T>0$ be a given time horizon, let $\eta^\eps$ be the structure displacement
component of the global strong solution from Theorem \ref{tm:main} and let $\psi^{\eps}$ be the
stream function defined by (\ref{def:stream_f}). Then the following 
estimates for the stream function hold:
\begin{align}
|\nabla\psi^{\eps}&(x,z,t)|  \leq C \left(|\partial_x^2\eta^{\eps}(x,t)|
+\frac{\eps}{\eta^{\eps}(x,t)}\right ),\quad\text{for all } (x,z,t)\in\Omega_{\eta^{\eps}}(t)\times(0,T)\,,\label{nabla:stream}\\
\|\partial_x\psi^{\eps} (t)&\|_{L^2(\Omega_{\eta^{\eps}}(t))}\leq C\eps^{3/2}\,,
\quad \text{for all } t\in (0,T)\,, 
\label{dx:stream}\\
\|\partial_y\psi^{\eps} (t)&\|_{L^2(\Omega_{\eta^{\eps}}(t))}\leq 
C \eps^{3/4}\left\|\frac{1}{\eta^{\eps}(t)}\right\|_{L^1(\omega)}^{1/4}\,,
\quad \text{for all } t\in (0,T)\,, \label{dy:stream}
\\
\|\partial_t\psi^{\eps}&\|_{L^2(\Omega_{\eta^{\eps}}(t)\times(0,T))}\leq C
\left(\eps^r + \eps^{-4}\int_0^T\|\partial^3_x\eta^{\eps}(s)\|_{L^2(\omega)}\dd s \right)^{1/2}\,, \label{dt:stream}
\\
\|\partial^2_x\psi^{\eps}&\|_{L^2(\Omega_{\eta^{\eps}}(t)\times(0,T))
}\leq C\left (\int_0^T \left(\eps\|\partial^3_x\eta^{\eps}(s)\|^2_{L^2(\omega)} 
+ \eps^{3/2}\|\partial^3_x\eta^{\eps}(s)\|^{3/2}_{L^2(\omega)}\right) \dd s   \right )^{1/2}. \label{ddx:stream}
\end{align}
\end{lemma}

\noindent The proof of Lemma \ref{StreamEstimates} is deferred to Appendix \ref{app:stream_est}.

\begin{proof}[Proof of Proposition \ref{prop:no_contact}]
Let us define the fluid test function $\vphi^{\eps}$ by 
$\vphi^{\eps}=\nabla^{\perp}\psi^{\eps}=(-\partial_y \psi^{\eps},\partial_x\psi^{\eps})$. Now we 
test equations \eqref{1.eq:stokes} and \eqref{1.eq:elast} with $\vphi^{\eps}$ 
and $\partial^2_x\eta^{\eps}$, respectively. Following the calculations in 
\cite[Equation (83)]{grandmont2016existence}) and rescaling time and data according to (S1)-(S2), 
we obtain (cf.~\cite[Equation (83)]{grandmont2016existence}): for a.e.~$t\in(0,T)$
\begin{align}\nonumber
&\frac{\updelta}{2\eps^{r}}\|\partial^2_x\eta^{\eps}(t)\|^2_{L^2(\omega)}
+\left\|\frac{6}{\eta^{\eps}(t)}\right\|_{L^1(\omega)}
+\frac{\upbeta}{\eps^3}\int_0^t\|\partial^3_x\eta^{\eps}(s)\|^2_{L^2(\omega)}\dd s
=\eps^3\uvro\int_0^t\|\partial_{tx}\eta^{\eps}(s)\|^2_{L^2(\omega)}\dd s\\
&+\eps^3\uvro\int_\omega\partial_t\eta^{\eps}(t)\partial^2_x\eta^{\eps}(t)\dd x \label{Distance1}
+ \left\|\frac{6}{\eta_0}\right\|_{L^1(\omega)}
+\eps^{-2}\underbrace{\int_0^t\int_{\Omega_{\eta}(s)}\partial^2_x\psi^{\eps}\big (\partial_x v^{\eps}_2-2\partial_y v^{\eps}_1\big )\dd\bx\dd s}_{I_1}\\
&+\underbrace{\int_0^t\int_{\Omega_{\eta}(s)}\big (\partial_t\bv^{\eps} + \eps^{-2}(\bv^{\eps}\cdot\nabla )\bv^{\eps})\cdot\vphi^{\eps}\dd\bx\dd s}_{I_2}
+\eps^{-2}\underbrace{\int_0^t\int_{\Omega_{\eta}(s)}\bs f\cdot\vphi^{\eps}\dd\bx\dd s}_{I_3}.\nonumber
\end{align}
In the sequel we estimate all terms on the right hand side. 
The first two terms can be estimated directly from \eqref{EIEkey}:
\begin{align}\nonumber
\eps^3\uvro\int_0^t\|\partial_{tx}\eta^{\eps}(s)\|^2_{L^2(\omega)}\dd s
&+\eps^3\uvro\int_\omega\partial_t\eta^{\eps}(t)\partial^2_x\eta^{\eps}(t)\dd x
\\ \label{FirstTwoTerms}
\leq C\eps^{r+2} 
&+ \eps^3\uvro\|\partial_t\eta^{\eps}(t)\|_{L^2(\omega)}\|\partial^2_x\eta^{\eps}(t)\|_{L^2(\omega)}
\leq C\eps^{r+2} + C\eps^2 \leq C\eps^2\,,
\end{align}
for $\eps$ small enough.
Moreover, by assumption on the initial data we have $\|6/\eta_0\|_{L^1(\omega)}=C/\eps$. 

The remainder of the proof consists of estimating integral terms $I_1$, $I_2$ and $I_3$,
where we repeatedly use the stream function estimates from Lemma \ref{StreamEstimates}.
First, utilizing (\ref{EIEkey}), (\ref{ddx:stream}) and the Young inequality, we have
\begin{align*}
|I_1| &\leq 
C\|\partial^2_x\psi^{\eps}\|_{L^2(0,t;L^2(\Omega_\eta))}\|\nabla\bv^{\eps}\|_{L^2(0,t;L^2(\Omega_\eta))}\\
&\leq C\eps^{3/2}\left (\int_0^t \left(\eps\|\partial^3_x\eta^{\eps}(s)\|^2_{L^2(\omega)} 
+ \eps^{3/2}\|\partial^3_x\eta^{\eps}(s)\|^{3/2}_{L^2(\omega)}\right) \dd s   \right )^{1/2}\\
&\leq C\eps^{3} + \kappa\eps\int_0^t \left(\|\partial^3_x\eta^{\eps}(s)\|^2_{L^2(\omega)} 
+ \eps^{1/2}\|\partial^3_x\eta^{\eps}(s)\|^{3/2}_{L^2(\omega)}\right) \dd s\\
& \leq C\eps^{3} 
+ \frac{7\eps}{4}\int_0^t\|\partial^3_x\eta^{\eps}(s)\|^2_{L^2(\omega)} \dd s\,.
\end{align*}

In order to estimate $I_2$, we split the integral into two parts. 
First, employing the Reynolds transport theorem, we estimate the term with the time derivative:
\begin{align*}
\left |\int_0^t\!\!\int_{\Omega_{\eta}(s)}\partial_t\bv^{\eps}\cdot\vphi^{\eps}\dd\bx\dd s\right |
&=\left |\frac{\dd}{\dd t}\left(\int_0^t\!\!\int_{\Omega_{\eta}(s)}\bv^{\eps}\cdot\vphi^{\eps}\dd\bx\dd s\right)
-\int_0^t\!\!\int_{\Omega_{\eta}(s)}\bv^{\eps}\cdot\partial_t\vphi^{\eps}\dd\bx\dd s\right.\\
&\quad \left. + \eps^{2}\int_0^t\!\!\int_\omega(\partial_t\eta^{\eps})^2\partial^2_x\eta^{\eps}\dd x\right | 
\leq \underbrace{\|\bv^{\eps}(t)\|_{L^2(\Omega_\eta(t))}\|\vphi^{\eps}(t)\|_{L^2(\Omega_\eta(t))}}_{I_{21}} \\
&\qquad + \underbrace{\left|\int_0^t\!\!\int_{\Omega_{\eta}(s)}\bv^{\eps}\cdot\partial_t\vphi^{\eps}\dd\bx\dd s\right |}_{I_{22}}
+ \eps^{2}\underbrace{\int_0^t\|\partial_t\eta^{\eps}\|^2_{L^{4}(\omega)}\|\partial^2_x\eta^{\eps}\|_{L^{2}(\omega)}\dd s}_{I_{23}}.
\end{align*}
In the following we estimate the obtained terms on the right hand side separately. 
Utilizing the energy estimate (\ref{EIEkey})
and the stream function estimates (\ref{dx:stream}) and (\ref{dy:stream}) we get
\begin{align*}
I_{21} &\leq \|\bv^{\eps}(t)\|_{L^2(\Omega_\eta(t))}\|\nabla\psi^{\eps}(t)\|_{L^2(\Omega_\eta(t))}\\
&\leq C\eps^{1/2}\left(\eps^{3/2} + \eps^{3/4}\left\|\frac{1}{\eta^{\eps}(t)}\right\|_{L^1(\omega)}^{1/4}\right)
\leq C\eps^{2} + C\eps^{5/3} + \frac14\left\|\frac{1}{\eta^{\eps}(t)}\right\|_{L^1(\omega)}.
\end{align*}
Using the definition $\vphi^{\eps}=\nabla^{\perp}\psi^{\eps}$ and integrating by parts we obtain
\begin{align*}
I_{22} &= \left|\int_0^t\!\!\int_{\Omega_{\eta}(s)}(v^{\eps}_2,-v^{\eps}_1)\cdot\nabla\partial_t\psi^{\eps}\dd\bx\dd s\right|\\
& = \left |\int_0^t\!\!\int_{\Omega_{\eta}(s)}(\partial_y v_1^{\eps}-\partial_xv_2^{\eps})\partial_t\psi^{\eps}\dd\bx\dd s 
+ \eps^{2}\int_0^t\!\!\int_\omega\partial_t\eta^{\eps}\partial_x\eta^{\eps}\partial_{tx}\eta^{\eps}\dd x  \right |\\
& \leq \|\nabla\bv^{\eps}\|_{L^2(0,t;L^2(\Omega_\eta))}\|\partial_t\psi^{\eps}\|_{L^2(0,t;L^2(\Omega_\eta))} 
+ \eps^{2}\int_0^t\|\partial_x\eta^{\eps}\|_{L^\infty(\omega)}\|\partial_t\eta^{\eps}\|_{L^{2}(\omega)}\|\partial_{tx}\eta^{\eps}\|_{L^2(\omega)}\dd s\\
&\leq C\eps^{3/2}\left(\eps^r + \eps^{-4}\int_0^t\|\partial^3_x\eta^{\eps}(s)\|_{L^2(\omega)}\dd s \right)^{1/2} \\
&\qquad + \eps^{2} \|\partial_x\eta^{\eps}\|_{L^\infty(0,t;L^\infty(\omega))}\|\partial_t\eta^{\eps}\|_{L^\infty(0,t;L^2(\omega))}\int_0^t\|\partial_{tx}\eta^{\eps}\|_{L^2(\omega)}\dd s\\
&\leq \frac{C}{\eps} + \eps^{r+4} + \int_0^t\|\partial^3_x\eta^{\eps}(s)\|_{L^2(\omega)}\dd s + C\eps^{(r+1)/2}\\
&\leq \frac{C}{\eps} +  C + \int_0^t\|\partial^3_x\eta^{\eps}(s)\|_{L^2(\omega)}^2\dd s\,
\end{align*}
for $\eps$ small enough.
Above we used the energy estimate (\ref{EIEkey}), the stream function estimate (\ref{dt:stream}) and the
Young inequality.
Combining the continuity of the Sobolev embedding $L^4(\omega)\hookrightarrow H^1(\omega)$ and the H\"older inequality, we find
\begin{align*}
I_{23}\leq C\int_0^t\|\partial_{tx}\eta^{\eps}\|^2_{L^2(\omega)}\|\partial^2_x\eta^{\eps}\|_{L^{2}(\omega)}\dd s
\leq C\|\partial^2_x\eta^{\eps}\|_{L^\infty(0,t;L^{2}(\omega))}\int_0^t\|\partial_{tx}\eta^{\eps}\|^2_{L^2(\omega)}\dd s
\leq C\eps^r\,.
\end{align*}

Next we estimate the convective term in $I_2$:
\begin{align*}
&\left|\int_0^t\!\!\int_{\Omega_{\eta}(s)}(\bv^{\eps}\cdot\nabla)\bv^{\eps}\cdot\vphi^{\eps}\dd\bx\dd s\right|\\
&\quad \leq\int_0^t\!\!\int_\omega \left (\int_0^{\eta^{\eps}(x,s)}|\bv^{\eps}|^2 \dd z\right)^{1/2}\left (\int_0^{\eta^{\eps}(x,s)}|\nabla \bv^{\eps}|^2\dd z \right)^{1/2}
\sup_z|\vphi^{\eps}|\,\dd x\dd s\\
&\quad \leq C\underbrace{\int_0^t\!\!\int_\omega \left (\int_0^{\eta^{\eps}(x,s)}|\bv^{\eps}|^2\dd z \right)^{1/2}\left (\int_0^{\eta^{\eps}(x,s)}|\nabla \bv^{\eps}|^2 \dd z\right)^{1/2}\left|\partial_x^2\eta^{\eps}(x,s)\right|\dd x\dd s}_{I_{24}}\\
&\qquad + C\underbrace{\int_0^t\!\!\int_\omega \left (\int_0^{\eta^{\eps}(x,s)}|\bv^{\eps}|^2 \dd z\right)^{1/2}\left (\int_0^{\eta^{\eps}(x,s)}|\nabla \bv^{\eps}|^2\dd z \right)^{1/2}\frac{\eps}{\eta^{\eps}(x,s)}\,\dd x\dd s}_{I_{25}},
\end{align*}
where we used the stream estimate (\ref{nabla:stream}).
Integral term $I_{24}$ is estimated as follows:
\begin{align*}
I_{24} &\leq \int_0^t\!\!\int_\omega\int_0^{\eta^{\eps}(x,s)}|\nabla\bv^{\eps}|^2 \dd z\dd x\dd s
+ \int_0^t\!\!\int_\omega\left(\int_0^{\eta^{\eps}(x,s)}|\bv^{\eps}|^2 \dd z\right)|\pa_x^2\eta^\eps|^2\dd x\dd s\\
& \leq C\eps^3 + \int_0^t \|\bv^{\eps}\|_{L^2(\Omega_{\eta}(s))}^2\|\pa_x^2\eta^\eps\|_{L^\infty(\omega)}^2\dd s
\leq C\eps^3 + C\|\bv^{\eps}\|_{L^\infty(0,t;L^2(\Omega_{\eta}(s)))}^2\int_0^t \|\pa_x^3\eta^\eps\|_{L^2(\omega)}^2\dd s\\
&\leq C\eps^3 + C\eps\int_0^t \|\pa_x^3\eta^\eps\|_{L^2(\omega)}^2\dd s\,,
\end{align*}
where we used the energy inequality (\ref{EIEkey}) and the continuity of the 1D Sobolev embedding
$L^{\infty}(\omega)\hookrightarrow H^1(\omega)$.
In order to estimate $I_{25}$, note that (cf.~\ref{2.ineq:eps_P})
\begin{equation*}
\int_0^{\eta^{\eps}(x,t)}|\bv^{\eps}|^2\dd z\leq \frac12\eta^{\eps}(x,t)^2\int_0^{\eta^{\eps}(x,t)}|\partial_z \bv^{\eps}|^2\dd z\,.
\end{equation*}
Employing the latter in $I_{25}$ we obtain
\begin{align*}
I_{25}\leq C\eps\|\nabla\bv^{\eps}\|_{L^2(0,t;L^2(\Omega_{\eta}(s)))}^2 \leq C\eps^4\,.
\end{align*}

Finally, we estimate the force term $I_3$:
\begin{align*}
|I_3| &\leq \|\bs f\|_{L^\infty(0,t;L^\infty(\Omega_{\eta}(s)))}\int_0^t\int_{\Omega_{\eta}(s)}|\vphi^{\eps}|\dd\bx\dd s
\leq C\left(\int_\omega \eta^{\eps}(x,t)\dd x\right)^{1/2}\|\nabla\psi^\eps\|_{L^\infty(0,t;L^2(\Omega_{\eta}(s)))}\\
&\leq C\eps^{1/2}\left(\eps^{3/2} + \eps^{3/4}\left\|\frac{1}{\eta^{\eps}}\right\|_{L^\infty(0,t;L^1(\omega))}^{1/4}\right)
\leq C\eps^{2} + \frac{\eps^{2}}{4}\left\|\frac{1}{\eta^{\eps}}\right\|_{L^\infty(0,t;L^1(\omega))}\,.
\end{align*}

Putting all together, for $\eps$ small enough, we obtain
\begin{equation*}
\left\|\frac{1}{\eta^{\eps}}\right\|_{L^\infty(0,T;L^1(\omega))}\leq \frac{C}{\eps}\,.
\end{equation*}
Combining the latter with the energy estimate \eqref{EIEkey} we get the uniform estimate
\begin{align}\label{DistanceL1H2}
\left\|\frac{\eps}{\eta^{\eps}}\right\|_{L^\infty(0,T;L^1(\omega))}
+\left\|\frac{\eta^{\eps}}{\eps} \right \|_{L^\infty(0,T;H^2_\#(\omega))}\leq C\,,
\end{align}
where the constant $C$ does not depend on $\eps$. 
Having at hand \eqref{DistanceL1H2}, we can invoke \cite[Proposition 7]{grandmont2016existence} 
to conclude that
\begin{align}\label{DistanceFinal}
\left\|\frac{\eps}{\eta^{\eps}}\right\|_{L^\infty(0,T;L^\infty(\omega))} \leq C\,.
\end{align}
This finishes the proof of Proposition \ref{prop:no_contact}.
\end{proof}

Due to nonlinearities which appear in the original model, 
in the course of the derivation of the reduced model, 
we will also need a strong convergence of the sequence of structure
displacements $(\eta^\eps)$. Therefore, 
a uniform estimate on the time derivative $(\pa_t\eta^\eps)$ will be in order 
(cf.~Proposition \ref{2:ineq:pateta}), but before that
we need uniform estimates on the fluid velocity.

\subsection{Uniform estimates for the fluid velocity}
In order to obtain uniform estimates for the fluid velocity, 
we need to re-write the system \eqref{1.eq:stokes}-\eqref{KinematicBC} on the fixed domain 
$\Omega\times (0,T)$, where $\Omega=\omega\times (0,1)\subset\R^2$. For that purpose 
we introduce the following change of spatial variables:
\begin{equation}\label{VarChange}
\left (
\begin{array}{c}
\hat x\\
\hat y
\end{array}
\right )=
\left (
\begin{array}{c}
x\\
\dfrac{z}{\eta^\eps(x,t)}
\end{array}
\right ),
\end{equation}
where new variables are denoted by hats.
The rescaled spatial gradient can be calculated as
\begin{equation}\label{RescaledNabla}
\nabla^{\eps}_{\eta}=
\left (
\begin{array}{c}
\partial_{\hat x}-\hat y\dfrac{\partial_{\hat x}\eta^{\eps}}{\eta^{\eps}}\partial_{\hat y}
\\[1.5em]
\dfrac{1}{\eta^{\eps}}\partial_{\hat y}
\end{array}
\right )
=\left (
\begin{array}{c}
\partial_{\hat x}-\hat y\dfrac{\partial_{\hat x}\hat{\eta^{\eps}}}{\hat{\eta^{\eps}}}\partial_{\hat y}
\\[1.5em]
\dfrac{1}{\eps\hat{\eta^{\eps}}}\partial_{\hat y}
\end{array}
\right ),
\end{equation}
where we, motivated by \eqref{DisplacementEstimate}, introduced
the rescaled displacement $\hat{\eta}^{\eps}(x, {t}) = \eps^{-1}\eta^{\eps}(x,t)$.
Jacobian of the spatial transformation (\ref{VarChange}) then equals $(\eps\hat{\eta^{\eps}})^{-1}$, 
while the Jacobian of the full space-time transformation (\ref{VarChange}) with the time scale (S2) equals 
$\eps(\hat{\eta^{\eps}})^{-1}$.

Writing down the fluid dissipation term from the energy dissipation inequality \eqref{EIEkey}
in transformed variables implies the following uniform estimate for the fluid velocity 
$\hat \bv^\eps(\hat x,\hat y, t) = \bv^\eps(x,z,t)$ on the fixed domain: for a.e.~$t\in(0,T)$
\begin{equation}\label{VelEstimate}
\int_0^{ t}\!\!\int_{\Omega}|\nabla_\eta^\eps\hat\bv^\eps|^2\hat\eta^\eps\,\dd \hat\bx\dd s
\leq C\eps^2\,.
\end{equation}
In particular, this implies
\begin{equation}\label{VelEstimate}
\int_0^{ t}\!\!\int_{\Omega}|\pa_y\hat\bv^\eps|^2\,\dd \hat\bx\dd s
\leq C\eps^4\,.
\end{equation}
Employing the Poincar\'e inequality (\ref{2.ineq:eps_P}) on the rescaled domain and 
estimate (\ref{DisplacementEstimate}) we find:
for a.e.~$t\in(0,T)$
\begin{equation*}
\int_0^{ t}\!\!\int_{\Omega}|\hat\bv^\eps|^2\hat\eta^\eps\,\dd \hat\bx\dd s
\leq 
C\eps^2\int_0^{ t}\!\!\int_{\Omega}|\nabla_\eta^\eps\hat\bv^\eps|^2\hat\eta^\eps\,\dd \hat\bx\dd s
\leq C\eps^4\,.
\end{equation*}
Therefore, rescaling the fluid velocity $\hat{\bv}^\eps$ according to
$\tilde\bv^{\eps} = \eps^{-2}\hat{\bv}^{\eps}$,
we obtain the following uniform estimates:
\begin{align}\label{3.est:payvel}
\|\pa_y\tilde\bv^{\eps}\|_{L^2(0,T;L^2(\Omega))} &\leq C\,,\\
\|\tilde\bv^{\eps}\sqrt{\hat\eta^{\eps}}\|_{L^2(0,T;L^2(\Omega))} &\leq C\,. \label{3.est:unif_vel}
\end{align}

\subsection{Uniform estimate for the pressure}
Let us write down the weak formulation of the
original FSI problem which also involves the pressure term and let us take all
scalings from above into account, but neglecting hats and tildas. 
Then for every test function $\vphi$ compactly supported in $\Omega\times(0,T)$ we have
\begin{align}\nonumber
\int_0^T\!\!\int_{\Omega} p^{\eps}(\nabla_{\eta}^{\eps}\cdot\vphi)\,\eta^{\eps}\dd\bx\dd t = 
-\eps^4\int_0^T\!\!\int_{\Omega}\bv^{\eps}\cdot\partial_{{t}}\vphi\,\eta^{\eps}\dd\bx\dd t
+ \eps^4\int_0^T\!\!\int_{\Omega}(\bv^{\eps}\cdot\nabla_\eta^{\eps})\bv^{\eps}\cdot\vphi \,\eta^{\eps}\dd\bx\dd t 
\\ \label{WeakRescaled_press}
 + 2\eps^2\int_0^T\!\!\int_{\Omega}\D_{\eta}^{\eps}(\bv^{\eps}):\D_{\eta}^{\eps}(\vphi)\,\eta^{\eps}\dd\bx\dd t
 - \int_0^T\!\!\int_{\Omega}\bs f^{\eps}\cdot\vphi \,\eta^{\eps}\dd\bx\dd t\,.
\end{align}
The structure terms vanish due to compact support of $\vphi$. 
We decompose the pressure functional as
\begin{align*}
P^\eps(\vphi)&:=
\int_0^T\!\!\int_{\Omega} p^{\eps}(\nabla_{\eta}^{\eps}\cdot\vphi)\,\eta^{\eps}\dd\bx\dd t =
P^{\eps}_{1}(\vphi) + P^{\eps}_{2}(\vphi)\,,
\end{align*} 
where
\begin{align*}
P^{\eps}_{1}(\vphi) &= -\eps^4\int_0^T\!\!\int_{\Omega}\bv^{\eps}\cdot\partial_{{t}}\vphi\,\eta^{\eps}\dd\bx\dd t\,,\\
P^{\eps}_{2}(\vphi) &= \eps^4\int_0^T\!\!\int_{\Omega}(\bv^{\eps}\cdot\nabla_\eta^{\eps})\bv^{\eps}\cdot\vphi \,\eta^{\eps}\dd\bx\dd t \\
 &\quad  + 2\eps^2\int_0^T\!\!\int_{\Omega}\D_{\eta}^{\eps}(\bv^{\eps}):\D_{\eta}^{\eps}(\vphi)\,\eta^{\eps}\dd\bx\dd t
 - \int_0^T\!\!\int_{\Omega}\bs f^{\eps}\cdot\vphi \,\eta^{\eps}\dd\bx\dd t\,.
\end{align*} 
Utilizing uniform estimates 
(\ref{DisplacementEstimate}), (\ref{3.est:no_contact}), (\ref{VelEstimate}), 
(\ref{3.est:unif_vel}) and assumption (S4) on the volume force, we obtain the following estimates:
\begin{align} \label{P1phi}
\left|P^{\eps}_{1}(\vphi)\right| 
&\leq 
\eps^4\|\eta^{\eps}\bv^{\eps}\|_{L^2(\Omega_T)}\|\partial_{{t}}\vphi\|_{L^2(\Omega_T)}
\leq C\eps^4 \|\vphi\|_{H^1_0(\Omega_T)}\,,
\\\nonumber
\left|P^{\eps}_{2}(\vphi)\right|&\leq 
\eps^4\left|\int_0^T\!\!\int_{\Omega}(\bv^{\eps}\cdot\nabla_\eta^{\eps})\bv^{\eps}\cdot\vphi \,\eta^{\eps}\dd\bx\dd t\right| 
+ 2\eps^2\|\eta^{\eps}\D_{\eta}^{\eps}(\bv^{\eps})\|_{L^2(\Omega_T)}\|\D_{\eta}^{\eps}(\vphi)\|_{L^2(\Omega_T)}\\
 & \qquad + \|\eta^{\eps}\bs f^{\eps}\|_{L^2(\Omega_T)}\|\vphi\|_{L^2(\Omega_T)}\label{P2phi}\\
&\leq C\|\vphi\|_{L^2(0,T;H_0^1(\Omega))}\,,\nonumber
\end{align}
where $\Omega_T\equiv \Omega\times(0,T)$ and $C>0$ is independent of $\eps$.
In this way we have proved the uniform boundedness of a sequence of functionals 
$(P^\eps)\subset H^{-1}(\Omega_T)$, 
\begin{equation}\label{3.est:unifpres}
\|P^\eps\|_{H^{-1}(\Omega_T)} \leq C\,.
\end{equation}
Moreover, we have proved the following uniform estimates: 
\begin{align}\label{P1}
\|P^{\eps}_{1}\|_{H^{-1}(\Omega_T)}\leq C\eps^4 \quad \text{and}\quad 
\|P^{\eps}_{2}\|_{L^2(0,T;H^{-1}(\Omega)}\leq C\,,
\end{align}
which (on a subsequence as $\eps\downarrow0$) imply 
\begin{align}
P^{\eps}_{1} &\to 0 \quad\text{strongly in } H^{-1}(\Omega_T)\,,\\
\label{P2}
P^{\eps}_{2}&\rightharpoonup P\quad \text{weakly in } L^2(0,T;H^{-1}(\Omega))\,.
\end{align}

\section{Derivation of the reduced model -- proof of Theorem \ref{tm:main}}\label{sec:RM}

In this section we derive the reduced model by letting $\eps\downarrow0$ and thus prove our main
result, Theorem \ref{tm:main}. The proof is devised into several steps. Based on the 
uniform estimates from the previous section we first identify weak (strong) limits, and then 
we identify relations between them. Finally, we interpret the obtained
reduced model as a weak formulation of the sixth-order thin-film type equation.  

\subsection{Weak (and strong) convergence results}

\begin{proposition}\label{2:ineq:pateta}
Let $(\eta^\eps)$ be a sequence of rescaled structure displacements, 
then there exists a constant $C>0$, independent of $\eps$, such that
\begin{equation}\label{dt:eta_eps}
\|\partial_t\eta^{\eps}\|_{L^2(0,T;H^{-1}_\#(\omega))} \leq C\,,
\end{equation}
where $H^{-1}_\#(\omega)$ denotes the dual of $H^{1}_\#(\omega)$. Moreover, the following estimate holds:
\begin{equation}\label{2:ineq:pateta1}
	\|\partial_t\eta^{\eps}\|_{L^2(0,T;H^{1}_\#(\omega))} \leq C\eps^{(r-3)/2}\,.
\end{equation}
\end{proposition}
\proof
The second inequality \eqref{2:ineq:pateta1} follows  directly the energy inequality (\ref{EIEkey}). However, notice that for $r<3$ it is not uniform in $\eps$ and therefore we need to prove unifrom estimate in weaker norm to obtain convergence of $\eta^{\eps}$ via Aubi-Lions lemma.

From the divergence free condition $\nabla^{\eps}_\eta\cdot \bv^{\eps}=0$ on $\Omega\times(0,T)$ 
we obtain
\begin{align*}
\partial_y{v}_2^{\eps}
= - \eps{\eta}^{\eps}\left (\partial_{x} -  y\dfrac{\partial_{ x}{\eta^{\eps}}}{{\eta^{\eps}}}\partial_{y} \right ){v}_1^{\eps}\,,
\end{align*}
which yields 
\begin{align*}
\left. v_2^{\eps}\right|_{\omega\times\{1\}} = 
- \eps{\eta}^{\eps}\int_0^1\left (\partial_{x} -  y\dfrac{\partial_{ x}{\eta^{\eps}}}{{\eta^{\eps}}}\partial_{y} \right ){v}_1^{\eps}\dd y\,.
\end{align*}
On the other hand, rescaling of the kinematic condition gives us
$\partial_{t}{\eta}^{\eps} = \eps^{-1}\left. v_2^{\eps}\right|_{\omega\times\{1\}}$, which provides
\begin{align*}
\int_0^T\!\!\int_\omega\partial_t\eta^{\eps}\varphi\, \dd x\dd t = 
\int_0^T\!\!\int_\omega {\eta}^{\eps}\!\int_0^1\left (\partial_{x} -  y\dfrac{\partial_{ x}{\eta^{\eps}}}{{\eta^{\eps}}}\partial_{y} \right ){v}_1^{\eps}\dd y\, \varphi \,\dd x\dd t
\end{align*} 
for every $\varphi\in L^2(0,T;H^1_\#(\omega))$. Integrating by parts in the latter identity 
and using uniform estimates (\ref{DisplacementEstimate}), (\ref{3.est:no_contact}) 
and (\ref{VelEstimate}) we find
\begin{equation*}
\left|\int_0^T\!\!\int_\omega\partial_t\eta^{\eps}\varphi\, \dd x\dd t\right|\leq C\|\varphi\|_{L^2(0,T;H^1_\#(\omega))}\,,
\end{equation*}
which implies (\ref{dt:eta_eps}).
\qed

Recall than the rescaled displacement is defined by
\begin{equation}\label{DispReScaled}
\hat{\eta}^{\eps}(x, {t}) = \eps^{-1}\eta^{\eps}(x,t)\,,
\end{equation}
where $\eta^{\eps}$ is the strong solution provided by Theorem \ref{ExistenceGlobal}.
Neglecting hats in further, uniform estimate \eqref{DisplacementEstimate} then gives
$\|\eta^{\eps}\|_{L^{\infty}(0,T;H^2_\#(\omega))}\leq C\,,$
which implies
\begin{equation}
\eta^{\eps} \overset{\ast}{\rightharpoonup} h\quad\text{weakly$^*$ in }L^{\infty}(0,T;H^2_\#(\omega))\,
\end{equation}
on a subsequence as $\eps\downarrow0$, and the uniform estimate (\ref{dt:eta_eps}) yields
\begin{equation}\label{pateta_convergence}
\pa_t\eta^{\eps} \rightharpoonup \pa_th\quad\text{weakly in }L^{2}(0,T;H^{-1}_\#(\omega))\,.
\end{equation}
Furthermore, due to compact and continuous embeddings 
$H^2_\#(\omega)\hookrightarrow\hookrightarrow 
C^1_\#(\overline\omega)\hookrightarrow H^{-1}_\#(\omega)$, respectively, 
invoking the Aubin-Lions lemma we conclude the strong convergence result
\begin{equation}\label{strong_conv}
\eta^{\eps}\to h\quad \text{strongly in }C^{0}([0,T];C^1_\#(\overline\omega))\,,
\end{equation}
which will be essential in the subsequent analysis. Moreover, the uniform no-contact results
of Proposition \ref{prop:no_contact} gives that $h(x,t)\geq c > 0$ for all $(x,t)\in\omega\times(0,T)$,
which in addition implies
\begin{equation}\label{strong_conv-1}
(\eta^{\eps})^{-1}\to h^{-1}\quad \text{strongly in }C^{0}([0,T];C^1_\#(\overline\omega))\,.
\end{equation}

Uniform estimate on the fluid velocity (\ref{3.est:unif_vel}) provides
\begin{equation*}
\tilde\bv^{\eps}\sqrt{\eta^{\eps}}\rightharpoonup \bu\quad\text{weakly in }L^2(0,T;L^2(\Omega))
\end{equation*}
on a subsequence as $\eps\downarrow0$. Due to the strong convergence results (\ref{strong_conv}) 
and (\ref{strong_conv-1}) we conclude that the rescaled fluid velocity $\tilde\bv^{\eps}$ itself 
has the weak limit (on a subsequence as $\eps\downarrow0$), i.e.~neglecting tildas we have
\begin{equation}\label{weakvel}
\bv^{\eps}\rightharpoonup \bu h^{-1/2}=:\bv \quad\text{weakly in }L^2(0,T;L^2(\Omega))\,.
\end{equation}
Furthermore, uniform estimate (\ref{3.est:payvel}) implies
\begin{equation}\label{weakpayvel}
\pa_y\bv^{\eps}\rightharpoonup \bs g\quad\text{weakly in }L^2(0,T;L^2(\Omega))
\end{equation}
and because of uniqueness of the weak limit we have $\bs g = \pa_y \bv$.

Let us now carefully analyze the pressure functional.
Employing the test function of the form $\bs\varphi = (0,\varphi_2)$ 
and using estimate (\ref{3.est:unifpres}), we calculate
\begin{equation}\label{press_phi2}
|P^{\eps}(0,\varphi_2)|
=\left|\frac{1}{\eps}\int_0^T\!\!\int_{\Omega}p^{\eps}\partial_y\varphi_2\, \dd\bx\dd t\right|
\leq C\|\varphi_2\|_{H_0^{1}(\Omega_T)}.
\end{equation}
Therefore, we proved that
\begin{equation}\label{payp}
\partial_y p^{\eps}\to 0\quad\text{strongly in }H^{-1}(\Omega_T)\,.
\end{equation}
  
Utilizing an arbitrary test function of the form $\bs\varphi = (\varphi_1,0)$ and integrating 
by parts we obtain
\begin{equation*}
P^{\eps}(\varphi_1,0)=
-\int_0^T\!\!\int_{\Omega}\eta^{\eps}\partial_xp^{\eps}\varphi_1\, \dd \bx\dd t
+\int_0^T\!\!\int_{\Omega}\pa_yp^{\eps}y\partial_x\eta^{\eps}\varphi_1\, \dd \bx\dd t\,.
\end{equation*}
Therefore, using the decomposition of $P^{\eps}$ we can write
\begin{equation}\label{etapaxp}
\int_0^T\!\!\int_{\Omega}\eta^{\eps}\partial_xp^{\eps}\varphi_1\, \dd \bx\dd t
=\int_0^T\!\!\int_{\Omega}\pa_yp^{\eps}y\partial_x\eta^{\eps}\varphi_1\, \dd \bx\dd t
 - P^{\eps}_{1}(\varphi_1,0) - P^{\eps}_{2}(\varphi_1,0)\,.
\end{equation}
Now, for an arbitrary $\phi\in H^{1}_0(0,T;H^{2}_0(\Omega))$, defining $\varphi_1 = \phi/\eta^{\eps}$, the latter 
identity reads
\begin{equation}\label{eq:paxp}
\int_0^T\!\!\int_{\Omega}\partial_xp^{\eps}\phi\, \dd \bx\dd t
=\int_0^T\!\!\int_{\Omega}\pa_yp^{\eps}\frac{y\partial_x\eta^{\eps}}{\eta^{\eps}}\phi\, \dd \bx\dd t
 - P^{\eps}_{1}(\frac{\phi}{\eta^{\eps}},0) - P^{\eps}_{2}(\frac{\phi}{\eta^{\eps}},0)\,.
\end{equation}
Let us estimate terms on the right hand side.
Integrating by parts with respect to $y$ and using identity (\ref{press_phi2}) we find
\begin{align*}
\int_0^T\!\!\int_{\Omega}\pa_yp^{\eps}\frac{y\partial_x\eta^{\eps}}{\eta^{\eps}}\phi\, \dd \bx\dd t
= - \int_0^T\!\!\int_{\Omega}p^{\eps}\pa_y\left(\frac{y\partial_x\eta^{\eps}}{\eta^{\eps}}\phi\right)\, \dd \bx\dd t
= -\eps P^{\eps}\left(0, \frac{y\partial_x\eta^{\eps}}{\eta^{\eps}}\phi \right)\,.
\end{align*}
Using estimates (\ref{P1phi}) and (\ref{P2phi}) we obtain
\begin{align*}
\left|\int_0^T\!\!\int_{\Omega}\pa_yp^{\eps}\frac{y\partial_x\eta^{\eps}}{\eta^{\eps}}\phi\, \dd \bx\dd t\right|
&\leq \eps \left|P^{\eps}_1\left(0, \frac{y\partial_x\eta^{\eps}}{\eta^{\eps}}\phi \right)\right|
+ \eps \left|P^{\eps}_2\left(0, \frac{y\partial_x\eta^{\eps}}{\eta^{\eps}}\phi \right)\right|\\
&\leq C\eps^5\left\|\pa_t\left(\frac{\partial_x\eta^{\eps}}{\eta^{\eps}}\phi\right)\right\|_{L^2(\Omega_T)}
+ C\eps \left\|\frac{\partial_x\eta^{\eps}}{\eta^{\eps}}\phi\right\|_{L^2(0,T;H_0^1(\Omega))}\\
&\leq C\eps^5\|\pa_{tx}\eta^{\eps}\|_{L^2(0,T;L^2(\omega))}\|(\eta^{\eps})^{-1}\|_{L^\infty(0,T;L^\infty(\omega))}\|\phi\|_{L^\infty(\Omega_T)} \\
&\quad + C\eps^5\left\|\frac{\partial_x\eta^{\eps}}{(\eta^{\eps})^2}\right\|_{L^\infty(0,T;L^\infty(\omega))}\|\pa_t\eta^{\eps}\|_{L^\infty(0,T;L^2(\omega))}\|\phi\|_{L^2(0,T;L^\infty(\Omega))} \\
&\quad + C\eps^5\left\|\frac{\partial_x\eta^{\eps}}{\eta^{\eps}}\right\|_{L^\infty(0,T;L^\infty(\omega))}\|\pa_t\phi\|_{L^2(\Omega_T)}\\
&\quad + C\eps \left\|\frac{\partial_x\eta^{\eps}}{\eta^{\eps}}\right\|_{L^\infty(0,T;L^\infty(\omega))}\|\phi\|_{L^2(0,T;H_0^1(\Omega))}\\
&\leq C\eps\left(\eps^{(r+5)/2} + \eps + \eps^4 + 1\right)\|\phi\|_{H^{1}_0(0,T;H^{2}_0(\Omega))}\,.
\end{align*}
Therefore,
\begin{equation}\label{paxp:1}
\pa_yp^{\eps}\frac{y\partial_x\eta^{\eps}}{\eta^{\eps}} \to 0\quad\text{strongly in } H^{-1}(0,T;H^{-2}(\Omega))\,.
\end{equation}
Above we also used uniform estimates (\ref{DisplacementEstimate}), (\ref{3.est:no_contact}) and \eqref{2:ineq:pateta1}. 
Next,
\begin{align*}
\left|P^{\eps}_{1}(\frac{\phi}{\eta^{\eps}},0)\right| &\leq C\eps^4\left\|\pa_t\left(\frac{\phi}{\eta^\eps}\right)\right\|_{L^2(\Omega_T)}\\
&\leq C\eps^4\|(\eta^{\eps})^{-1}\|_{L^\infty(0,T;L^\infty(\omega))}\|\pa_t\phi\|_{L^2(\Omega_T)}\\
&\quad + C\eps^4\|(\eta^{\eps})^{-2}\|_{L^\infty(0,T;L^\infty(\omega))}\|\pa_t \eta^{\eps}\|_{L^2(0,T;L^\infty(\omega))}\|\phi\|_{L^\infty(0,T;L^2(\Omega))}\\
&\leq C\eps\left(\eps^{(r+3)/2} + \eps^3\right)\|\phi\|_{H^{1}_0(\Omega_T)}\,,
\end{align*}
which implies 
\begin{equation}\label{paxp:2}
P^{\eps}_{1}(\frac{\cdot}{\eta^{\eps}},0) \to 0\quad\text{strongly in } H^{-1}(0,T;H^{-1}(\Omega))\,.
\end{equation}
Finally,
\begin{align*}
\left|P^{\eps}_{2}(\frac{\phi}{\eta^{\eps}},0)\right| &\leq C\left\|\frac{\phi}{\eta^\eps}\right\|_{L^2(0,T;H_0^1(\Omega))}\leq C\left\|\phi\right\|_{L^2(0,T;H_0^1(\Omega))}\,.
\end{align*}
The latter implies 
\begin{equation}\label{paxp:3}
P^{\eps}_{2}(\frac{\cdot}{\eta^{\eps}},0) \rightharpoonup q\quad\text{weakly in } L^{2}(0,T;H^{-1}(\Omega))\,
\end{equation}
for some $q\in L^{2}(0,T;H^{-1}(\Omega))$.
Putting together (\ref{paxp:1})-(\ref{paxp:3}) we have
\begin{align}\label{paxtoq}
\pa_xp^{\eps} \rightharpoonup -\, q \quad\text{weakly in } H^{-1}(0,T;H^{-2}(\Omega))\,.
\end{align}
Note from (\ref{eq:paxp}) that for an arbitrary test function 
$\xi\in H^{1}_0(0,T;H^{2}_0(\Omega))$ satisfying $\int_\Omega\xi \dd\bx = 0$, if we define
$\phi = \int_0^x\xi$, then
\begin{equation}\label{eq:p0val}
\int_0^T\!\!\int_{\Omega}p^{\eps}\xi\, \dd \bx\dd t 
= - \int_0^T\!\!\int_{\Omega}\pa_yp^{\eps}\frac{y\partial_x\eta^{\eps}}{\eta^{\eps}}\phi\, \dd \bx\dd t
 + P^{\eps}_{1}(\frac{\phi}{\eta^{\eps}},0) + P^{\eps}_{2}(\frac{\phi}{\eta^{\eps}},0)\,.
\end{equation}
Now repeating the estimates as above we conclude that there exists $p\in L^{2}(0,T;H^{-1}(\Omega))$
such that
\begin{align*}
p^{\eps} - \pi^{\eps}  \rightharpoonup p \quad\text{weakly in } H^{-1}(0,T;H^{-2}(\Omega))\,,
\end{align*}
where
$\pi^{\eps}(t) = \int_\Omega p^{\eps}\dd\bx$. 

Let us now estimate the mean value of the pressure $\pi^{\eps}$. From the weak 
formulation of the rescaled FSI problem (cf.~(\ref{WeakRescaled}) below) we have
\begin{align}\label{mvpress}
\int_0^T\!\!\int_{\Omega}p^{\eps}\pa_y \varphi_2\, \dd \bx\dd t &= \eps P^{\eps}(0,\varphi_2) 
-\uvro\eps^6\int_0^T\!\!\int_\omega\partial_{t}\eta^{\eps}\partial_{t}\psi\,\dd x\dd t\\
&\quad -\updelta\eps^{3-r}\int_0^T\!\!\int_\omega\partial_{t}\eta^{\eps}\partial_{x}^2\psi\nonumber
\,\dd x\dd t
+\upbeta\int_0^T\!\!\int_0^L\partial_{x}^2\eta^{\eps}\partial_{x}^2\psi
\,\dd x\dd t\,.
\end{align}
For an arbitrary $\zeta\in H_0^1(0,T)$, 
taking $\varphi_2 = y\zeta$ and $\psi = \zeta$ in (\ref{mvpress}) we obtain 
\begin{align*}
\int_0^T\!\!\pi^{\eps}\zeta\,\dd t &= \eps P^{\eps}(0,y\zeta) 
-\uvro\eps^6\int_0^T\!\!\int_\omega\partial_{t}\eta^{\eps}\partial_{t}\zeta\,\dd x\dd t
\end{align*}
Repeating the above estimates once again we find that
\begin{align*}
\pi^{\eps}  \to 0 \quad\text{strongly in } H^{-1}(0,T)\,.
\end{align*}
Therefore, we have
\begin{align*}
p^{\eps}  \rightharpoonup p \quad\text{weakly in } H^{-1}(0,T;H^{-2}(\Omega))\,
\end{align*}
on a subsequence as $\eps\downarrow0$.
From (\ref{paxp:3}), (\ref{paxtoq}) and (\ref{eq:p0val}) we
readily conclude that $\pa_x p = -q$. Hence, having $p\in L^{2}(0,T;H^{-1}(\Omega))$ and 
$\nabla p\in L^{2}(0,T;H^{-1}(\Omega))$ (recall (\ref{payp})), the Lions-Ne\v cas lemma yields
$p\in L^{2}(0,T;L^{2}(\Omega))$.

Going back to (\ref{etapaxp}) and estimating terms on the right hand side like above we obtain
\begin{equation*}
\eta^{\eps}\pa_xp^{\eps} \rightharpoonup - P \quad\text{weakly in } H^{-1}(0,T;H^{-2}(\Omega))\,,
\end{equation*}
where $P$ is defined in (\ref{P2}). On the other hand, the above convergence results allow us 
to conclude that
\begin{equation*}
\eta^{\eps}\pa_xp^{\eps} \rightharpoonup h\pa_xp \quad\text{weakly in } H^{-1}(0,T;H^{-2}(\Omega))\,
\end{equation*}
with $h\pa_xp\in L^{2}(0,T;H^{-1}(\Omega))$. Thus, 
\begin{equation}
P = -h\pa_xp\,.\label{Phpaxp}
\end{equation}

\subsection{Identification of the reduced model}
Recalling that we work with transformed variables and rescaled unknowns, 
let us write down the weak formulation of the
original FSI problem (\ref{1.eq:stokes})-(\ref{KinematicBC}), which also involves the pressure term:
\begin{align}\nonumber
-\eps^3\int_0^T\!\!\int_{\Omega}\bv^{\eps}\cdot\partial_{{t}}\vphi\,\eta^{\eps}\dd\bx\dd t
+ \eps^3\int_0^T\!\!\int_{\Omega}(\bv^{\eps}\cdot\nabla_\eta^{\eps})\bv^{\eps}\cdot\vphi \,\eta^{\eps}\dd\bx\dd t
\\
\label{WeakRescaled}
+2\eps\int_0^T\!\!\int_{\Omega}\D_{\eta}^{\eps}(\bv^{\eps}):\D_{\eta}^{\eps}(\vphi)\,\eta^{\eps}\dd\bx\dd t
-\frac{1}{\eps}\int_0^T\!\!\int_{\Omega} p^{\eps}(\nabla_{\eta}^{\eps}\cdot\vphi)\,\eta^{\eps}\dd\bx\dd t
\\\nonumber
-\uvro\eps^4\int_0^T\!\!\int_\omega\partial_{t}\eta^{\eps}\partial_{t}\psi\,\dd x\dd t
-\updelta\eps^{1-r}\int_0^T\!\!\int_\omega\partial_{t}\eta^{\eps}\partial_{x}^2\psi
\,\dd x\dd t
+\frac{\upbeta}{\eps^2}\int_0^T\!\!\int_0^L\partial_{x}^2\eta^{\eps}\partial_{x}^2\psi
\,\dd x\dd t
\\\nonumber
=\frac{1}{\eps}\int_0^T\!\!\int_{\Omega}\bs f^{\eps}\cdot\vphi \,\eta^{\eps}\dd\bx\dd t
\end{align}
where the test functions $(\vphi,\psi)\in C^1_c([0,T);\mathcal{V}_F(t)\times H^2_{\#}({\omega}))$ satisfy $\left.\vphi\right|_{\omega\times\{1\}}=(0,\psi)$.

Notice that the highest order terms in \eqref{WeakRescaled} are of order $\eps^{-2}$.
Thus, multiplying \eqref{WeakRescaled} by $\eps^2$ and taking test functions of the form 
$(\vphi,0)$, where $\vphi$ has compact support in $\Omega\times (0,T)$, in the limit (on a 
subsequence as $\eps\downarrow0$) we obtain only the pressure term
\begin{equation}\label{limit:p_y}
\int_0^T\!\!\int_{\Omega}p\,\partial_{y}\varphi_2\,\dd\bx\dd t = 0\,.
\end{equation}
Limit equation (\ref{limit:p_y}) implies that the pressure $p$ is
independent of the vertical variable $y$.
Multiplying again \eqref{WeakRescaled} by $\eps^2$ and taking arbitrary test function 
$(\vphi,\psi)$ compactly supported in $(0,T)$
such that $\left.\vphi\right|_{\omega\times\{1\}}=(0,\psi)$, in the limit (on a 
subsequence as $\eps\downarrow0$) we obtain
\begin{align}\label{p_integral}
\chi_{\{r=3\}}\updelta\int_0^T\!\!\int_\omega h\partial_{{t}}\partial_{x}^2\psi \,\dd x\dd t 
+ \upbeta\int_0^T\!\!\int_\omega\partial_{x}^2 h\partial_{x}^2\psi \,\dd x\dd t
 = \int_0^T\!\!\int_\omega p\psi\,\dd x\dd t\,,
\end{align}
where $\chi_{\{r=3\}} = 1$ if $r=3$, and $0$ otherwise. 
More precisely, the integral on the right hand side
comes as follows:
\begin{align*}
\int_0^T\!\!\int_\Omega p\,\pa_y\varphi_2\, \dd \bx\dd t 
= \int_0^T\!\!\int_\omega p\int_0^1\pa_y\varphi_2\, \dd y\, \dd x\dd t 
= \int_0^T\!\!\int_\omega p\psi\,\dd x\dd t\,,
\end{align*}
where we used that $p$ is independent of $y$ and $\psi(x,t) = \int_0^1\pa_y\varphi_2(x,y,t)\dd y$.
We can write equation (\ref{p_integral}) formally in the sense of distributions as
\begin{align}\label{PresDis}
-\chi_{\{r=3\}}\updelta\partial_{t}\partial^2_{x}h + \upbeta\partial^4_{x}h = p\,,
\end{align}
which tells us that in the lubrication approximation regime (on the limit as $\eps\downarrow0$)
the pressure is balanced by the structure bending and viscosity of the structure if 
the latter is large enough.

Next, we derive the equation for the limit velocity $\bv$. 
The divergence free equation on the reference domain reads
\begin{equation}\label{eq:divfree_ref}
\pa_x v_1^\eps - y\frac{\pa_x\eta^\eps}{\eta^\eps}\pa_y v_1^\eps 
+ \frac{1}{\eps\eta^\eps}\pa_yv_2^\eps = 0\,.
\end{equation}
Testing (\ref{eq:divfree_ref}) by $\eps\varphi$, where $\varphi$ is a test function compactly supported
in $\Omega$, and integrating by parts we obtain 
\begin{equation}\label{eq:divfree_ref_int}
-\eps\int_0^T\!\!\int_\Omega v_1^\eps\pa_x\varphi\,\dd\bx\dd t - 
\eps\int_0^T\!\!\int_\Omega y\frac{\pa_x\eta^\eps}{\eta^\eps} \pa_y v_1^\eps  \varphi\,\dd\bx\dd t
+ \int_0^T\!\!\int_\Omega\frac{1}{\eta^\eps}\pa_yv_2^\eps\,\varphi\,\dd\bx\dd t = 0\,.
\end{equation}
Invoking convergence results (\ref{strong_conv})-(\ref{weakpayvel}) we can pass to the limit 
in (\ref{eq:divfree_ref_int}) (on a subsequence as $\eps\downarrow0$) and thus obtain
$\int_0^T\!\!\int_\Omega h^{-1}\pa_yv_2\,\varphi\,\dd\bx\dd t = 0$, from which we
conclude that $h^{-1}\pa_yv_2 = 0$. Since $v_2 \in L^2(0,T;L^2(\Omega))$ and 
$\pa_y v_2 \in L^2(0,T;L^2(\Omega))$, function $v_2$ has the trace on $\omega$, and continuity
of the trace operator implies that $v_2 = 0$ on $\omega$. Furthermore, from the basic identity
$\displaystyle v_2(\cdot,y,\cdot) = \int_0^y\pa_y v_2 (\cdot,\zeta,\cdot) \dd \zeta$ we eventually conclude
that $v_2 = 0$.

Multiplying \eqref{WeakRescaled} by $\eps$ and taking $((\varphi_1,0),0)$ as a test function, where 
$\varphi_1$ has compact support in $\Omega$, we obtain (on the limit as $\eps\downarrow0$)
\begin{equation}\label{Lubrication}
\int_0^T\!\!\int_{\Omega}\frac{1}{h}\partial_{y}v_1\partial_{y}\varphi_1 \,\dd\bx\dd t 
-\int_0^T\!\!\int_{\Omega}\left(p h\, \partial_{x} \varphi_1 + p\pa_x h\,\varphi_1\right)\dd\bx\dd t 
  = \int_0^T\!\!\int_{\Omega}hf_1\varphi_1 \,\dd\bx\dd t \,.
\end{equation}
Integrating by parts, we obtain the following identity in the sense of distributions
\begin{equation}\label{Lubrication2}
-\int_0^T\!\!\int_{\Omega}\frac{1}{h}\partial_{y}^2v_1\varphi_1 \,\dd\bx\dd t 
+\int_0^T\!\!\int_{\Omega}h\pa_x p\, \varphi_1 \dd\bx\dd t 
  = \int_0^T\!\!\int_{\Omega}hf_1\varphi_1 \,\dd\bx\dd t \,,
\end{equation}
which can be written in the sense of equation in $L^2(0,T;H^{-1}(\Omega))$ as
\begin{equation}\label{FluidLimitEq}
\partial_y^2 v_1 = h^2\left(\partial_x p - f_1\right).
\end{equation}
Since the pressure $p$ and the displacement $h$ on the right hand side of \eqref{FluidLimitEq} are 
independent of $y$, the distributional equation can be solved explicitly in terms of $y$.
Again since $v_1 \in L^2(0,T;L^2(\Omega))$ and 
$\pa_y v_1 \in L^2(0,T;L^2(\Omega))$, function $v_1$ has the trace on $\omega$ and $\omega\times\{1\}$,
and continuity of the trace operator implies that $v_1 = 0$ on $\omega$ and and on $\omega\times\{1\}$.
Thus, $v_1$ inherits the no-slip boundary conditions from $v_1^\eps$, and
the explicit solution of (\ref{FluidLimitEq}) is given by
\begin{equation}\label{LimitVelFormula}
v_1(\cdot,y,\cdot)=\frac{1}{2}y(y - 1)h^2\partial_x p + h^2F(\cdot,y,\cdot)\,,
\end{equation}
where $\displaystyle F(\cdot,y,\cdot) 
= (y-1)\int_0^1\zeta f_1(\cdot,\zeta,\cdot)\dd\zeta 
- \int_y^1(y-\zeta)f_1(\cdot,\zeta,\cdot)\dd\zeta$.

Testing the divergence free equation (\ref{eq:divfree_ref}) 
with a test function $\varphi$ depending only on $x$, integrating by parts and 
employing the rescaled kinematic condition
$\left. v_2^{\eps}\right|_{\omega\times\{1\}} = \eps\partial_{t}{\eta}^{\eps}$ we obtain
\begin{align*}
0 = \int_0^T\!\!\int_\Omega\left(\pa_x v_1^\eps - y\frac{\pa_x\eta^\eps}{\eta^\eps}\pa_y v_1^\eps \right.
&+ \left.\frac{1}{\eps\eta^\eps}\pa_yv_2^\eps\right)\varphi\,\dd \bx\dd t\\
= & - \int_0^T\!\!\int_\Omega\left(v_1^\eps\pa_x\varphi 
- \frac{\pa_x\eta^\eps}{\eta^\eps} v_1^\eps\varphi \right)\dd \bx\dd t
+\int_0^T\!\!\int_\omega\frac{\partial_{t}\eta^{\eps}}{\eta^{\eps}}\varphi\,\dd x\dd t\,.
\end{align*}
Employing convergence results (\ref{pateta_convergence})-(\ref{weakvel}), the latter in the limit  
as $\eps\downarrow0$ becomes
\begin{equation*}
- \int_0^T\!\!\int_\Omega\left(v_1\pa_x\varphi 
- \frac{\pa_x h}{h} v_1\varphi \right)\dd \bx\dd t
+\int_0^T\!\!\int_\omega\frac{\partial_{t}h}{h}\varphi\,\dd x\dd t = 0\,,
\end{equation*}
where the second integral should be understood in the sense of distributions.
Since $\varphi$ is independent of $y$, the first integral can be written as
\begin{equation}\label{lubrik2}
- \int_0^T\!\!\int_\omega\left(\int_0^1v_1\dd y\,\pa_x\varphi 
- \frac{\pa_x h}{h} \int_0^1v_1 \dd y\,\varphi \right)\dd x\dd t
+\int_0^T\!\!\int_\omega\frac{\partial_{t}h}{h}\varphi\,\dd x\dd t = 0\,,
\end{equation}
and according to (\ref{LimitVelFormula}) we can calculate
\begin{align}\label{limitvelint}
\int_0^1 v_1 \dd y = -\frac{1}{12}h^2\pa_x p + h^2\Phi\,,
\end{align}
where 
$\displaystyle \Phi(x,t) = \int_0^1 F(x,y,t)\dd y$.
Moreover, (\ref{limitvelint}) implies that $h^2\pa_x p\in L^2(0,T;L^2(\omega))$ and therefore
$\pa_x p\in L^2(0,T;L^2(\omega))$.
Hence, going back to (\ref{lubrik2}) we find
\begin{equation*}
\int_0^T\!\!\int_\omega\left(\left(\frac{1}{12}h^2\pa_x p - h^2\Phi\right)\pa_x\varphi 
- h\pa_x h\left(\frac{1}{12} \pa_x p -  \Phi\right)\varphi \right)\dd x\dd t
+\int_0^T\!\!\int_\omega\frac{\partial_{t}h}{h}\varphi\,\dd x\dd t = 0\,.
\end{equation*}
Integrating by parts formally and taking $\tilde\varphi = \varphi/h$ as a new test function
we get
\begin{equation*}
-\int_0^T\!\!\int_\omega\left(h\pa_x\left(h^2\left(\frac{1}{12}\pa_x p - \Phi\right)\right)
+ \pa_x hh^2\left(\frac{1}{12} \pa_x p -  \Phi\right) \right)\tilde\varphi\,\dd x\dd t
+\int_0^T\!\!\int_\omega\partial_{t}h\,\tilde\varphi\,\dd x\dd t = 0\,,
\end{equation*}
which, noticing the product rule and neglecting tilda, can be written as
\begin{equation}\label{eq:Reynolds_int2}
-\int_0^T\!\!\int_\omega\pa_x\left(h^3\left(\frac{1}{12}\pa_x p - \Phi\right)\right)\varphi\,\dd x\dd t
+\int_0^T\!\!\int_\omega\partial_{t}h\,\varphi\,\dd x\dd t = 0\,
\end{equation}
for all $\varphi\in C^1_c([0,T);H^1_{\#}({\omega}))$.
This equation can be understood in the sense of equality
in space $L^2(0,T;H^{-1}(\Omega))$. Writing briefly
\begin{equation}\label{eq:Reynolds}
\pa_t h = \pa_x\left(h^3\left(\frac{1}{12}\pa_xp - \Phi \right) \right)\,,
\end{equation}
it can be interpreted as a Reynolds type equation for the pressure.
Integrating by parts in (\ref{eq:Reynolds_int2}) both in space and time we find its
weak formulation:
\begin{equation}\label{eq:Reynolds_int}
\int_0^T\!\!\int_\omega h^3\left(\frac{1}{12}\pa_x p - \Phi\right)\pa_x\varphi\,\dd x\dd t
-\int_0^T\!\!\int_\omega h\,\partial_{t}\varphi\,\dd x\dd t = 0\,,
\end{equation}
which holds for all $\varphi\in C^1_c([0,T);H^1_{\#}({\omega}))$. 
Notice that (\ref{eq:Reynolds_int}) is the second equation that relates $p$ and $h$, 
which together with (\ref{p_integral}) makes the system closed. Thus, the system 
(\ref{p_integral}), (\ref{eq:Reynolds_int}) can already 
be considered as a reduced model for the
FSI problem (\ref{1.eq:stokes})-(\ref{KinematicBC}), but we can make a step further.

\subsection{The sixth-order thin-film equation}
If we formally plug-in (\ref{PresDis}) into (\ref{eq:Reynolds}) we immediately obtain
the sixth-order thin-film equation
\begin{align}\label{ThinFilm}
\pa_t h = \pa_x\left(h^3\left(\frac{\upbeta}{12}\partial^5_{x}h 
- \chi_{\{r=3\}}\frac{\updelta}{12}\partial^3_{x}\partial_{t}h - \Phi \right) \right)\,,
\end{align}
which is to be understood in the sense of $L^2(0,T;H^{-1}_\#(\omega))$.
Recall from above that $\pa_x p\in L^2(0,T;L^2(\omega))$. Thus, 
(\ref{PresDis}) yields
$\upbeta\partial^5_{x}h -\chi_{\{r=3\}}\updelta\partial_{t}\partial^3_{x}h
\in L^2(0,T;L^2(\omega))$, and equation (\ref{ThinFilm}) (cf.~eq.~(\ref{eq:Reynolds_int}))
in its weak formulation becomes
\begin{equation}\label{eq:thinfilm_int}
\int_0^T\!\!\int_\omega h^3\left(\frac{\upbeta}{12}\partial^5_{x}h 
- \chi_{\{r=3\}}\frac{\updelta}{12}\partial^3_{x}\partial_{t}h - \Phi\right)\pa_x\varphi\,\dd x\dd t
-\int_0^T\!\!\int_\omega h\,\partial_{t}\varphi\,\dd x\dd t = 0\,,
\end{equation}
for all $\varphi\in C^1_c([0,T);H^1_{\#}({\omega}))$. In this way $h$ can be interpreted as a 
weak solution to the thin-film equation (\ref{ThinFilm}) with periodic boundary conditions
and initial data $h(0) = h_0$ (cf.~Theorem \ref{ExistenceGlobal}). 
However, 
following the pioneering work on higher-order evolution equations \cite{BeFr90},
much more can be proved. We summarize results for $h$ in the following:
\begin{proposition}
Function $h(x,t)$ is unique positive classical solution of equation (\ref{ThinFilm}).
Furthermore, $h$ is Lipschitz continuous in $x$ and H\"older continuous in $t$ 
with exponent $1/5$.
\end{proposition}

\begin{proof}
Strict positivity of $h$, i.e.~$h(x,t) \geq c >0$, has been already observed and it 
follows from Proposition \ref{prop:no_contact}. The Lipschitz continuity
follows from the fact that $h(\cdot,t)\in C^{0}([0,T];C^1_\#(\overline\omega))$
for all $t\in [0,T]$, while the H\"older continuity in $t$ with exponent $1/5$ can be proved
by the straightforward adjustment of the proof of \cite[Lemma 2.1]{BeFr90} for 
the sixth-order equation (\ref{eq:thinfilm_int}). 
Since $h$ is strictly positive and continuous, the standard parabolic regularity theory \cite{Eid69} 
applies and we conclude that all derivatives $\pa_th$, $\pa_xh$, \ldots, 
$\pa_x^6$, $\pa_x^3\pa_t h$, $\pa_x^4\pa_t h$ are continuous, 
which means that $h$ is a classical solution.
Finally, uniqueness follows by utilizing the strict positivity of $h$ and 
adopting the proof of \cite[Theorem 4.1 (iii)]{BeFr90} to our equation.
\end{proof}

In this way the sixth-order thin-film equation (\ref{ThinFilm}) can be seen as the reduced 
model of the FSI problem (\ref{1.eq:stokes})-(\ref{KinematicBC}) in the lubrication approximation regime,
i.e.~in the regime of the vanishing relative fluid thichness.
This finishes the proof of Theorem \ref{tm:main}.

\begin{remark}
An intimate relation between the FSI system (\ref{1.eq:stokes})-(\ref{KinematicBC}) and 
the thin-film equation (\ref{ThinFilm}) reveals in the following. Taking the time derivative of 
$\int_\omega h(t)^{-1} \dd x$ along solutions to (\ref{ThinFilm}), integrating by parts and 
employing the Poincar\'e inequality we find:
\begin{align*}
\frac{\dd}{\dd t}\int_\omega\frac{1}{h(t)}\dd x 
&= -\int_\omega \frac{1}{h^2}\pa_x\left(h^3\left(\frac{\upbeta}{12}\partial^5_{x}h 
- \chi_{\{r=3\}}\frac{\updelta}{12}\partial^3_{x}\partial_{t}h - \Phi \right) \right)\dd x\\
&= -\frac{\upbeta}{6}\int_\omega (\pa_x^3h)^2\dd x 
- \chi_{\{r=3\}}\frac{\updelta}{12} \frac{\dd}{\dd t}\int_\omega (\pa_x^2h)^2\dd x
+ 2\int_\omega\Phi \pa_xh\,\dd x\\
&\leq -\frac{\upbeta}{12}\int_\omega (\pa_x^3h)^2\dd x 
- \chi_{\{r=3\}}\frac{\updelta}{12} \frac{\dd}{\dd t}\int_\omega (\pa_x^2h)^2\dd x + C\,,
\end{align*}
which integrating from $0$ to $t$ ($t < T$) yields
\begin{equation*}
\chi_{\{r=3\}}\frac{\updelta}{2} \int_\omega (\pa_x^2h)^2\dd x + 
\int_\omega\frac{6}{h(t)}\dd x 
+ \frac{\upbeta}{2}\int_0^t\!\!\int_\omega (\pa_x^3h)^2\dd x\dd s \leq C\,.
\end{equation*}
Now observe that this estimate is equivalent to the uniform estimate of (\ref{Distance1})
after employing the lower-semicontinuity and passing to the limit in (\ref{Distance1}) as
$\eps\downarrow0$.
\end{remark}

\section{Conclusion and perspectives} \label{sec:CP}

Starting from the FSI problem(\ref{1.eq:stokes})-(\ref{KinematicBC}) under 
certain scaling assumptions, we have rigorously derived the sixth-order thin-film equation 
on the limit of the vanishing relative fluid thickness $\eps$. 
The procedure essentially relies on the quantitative energy estimate and the uniform no-contact result. 

The viscoelastic term $\updelta \partial^2_{x}\pa_t\eta$ in the structure equation of the FSI problem is present mainly as a regularizing term, which guarantees the global well-posedness of strong solutions (cf.~\cite{grandmont2016existence}). 
Apparently, it does not play the same role for obtaining the uniform estimates, and if $\updelta$ 
is not large enough, i.e.~is smaller than $O(\eps^{-3})$, this term vanishes in the limit. 
On the other hand, when the viscoelasticity parameter $\updelta$ is of the order 
$O(\eps^{-3})$ and thus survives in the limit as $\eps\downarrow0$, 
the resulting reduced model which includes higher-order 
spatio-temporal derivatives seems to be new in the literature.
So it would be interesting to identify a real physical FSI system in such viscoelasticity regime 
and investigate validity and significance of the novel reduced model.

The main point of reduced models is to provide approximate solutions of the original 
problem. So let $h$ be the classical solution of the thin-film equation (\ref{ThinFilm})
with initial datum $h_0 > 0$ and let $\eps > 0$. Then approximate solutions $(\bva^\eps,\paa^\eps,\upeta^\eps)$ to the
FSI problem (\ref{1.eq:stokes})-(\ref{KinematicBC}) on the reference domain $\Omega\times(0,T)$
with initial displacement $\eta^\eps_0 = \eps h_0$
can be reconstructed according to (\ref{PresDis}) and
(\ref{LimitVelFormula}) as:
\begin{align*}
\upeta^\eps &= \eps h\,,\\
\paa^\eps & = \upbeta\partial^4_{x}h -\chi_{\{r=3\}}\updelta\partial_{t}\partial^2_{x}h\,,\\
\va^\eps_1 & = \frac{\eps^2}{2}y(y - 1)h^2\partial_x\paa^\eps + \eps^2h^2F\,,\quad \va^\eps_2 = 0\,,
\end{align*}
with $F$ given as in (\ref{LimitVelFormula}).
Theorem \ref{tm:main} provides only strong
and weak convergence results in corresponding spaces, but from the application point 
of view it is important to have a quantitative error estimate. 
In linear models \cite{BuMu21P,BuMu21R} we were able to obtain error estimates 
for approximate solutions in strong norms, and it is our next aim to extend these 
results to the nonlinear framework of this paper. This will be the subject of our future work.

\appendix

\section{~}

\subsection{Non-dimensionalization and scaling assumptions}\label{app:SA}

Our starting point are dimensional Navier-Stokes (NS) equations describing the flow of a 
viscous fluid in a twodimensional channel with deformable top boundary:
\begin{align*}
\varrho_f\left(\pa_t\bv+(\bv\cdot \nabla)\bv\right) - \diver\sigma_f(\bv,p) &= \bs f\,,\quad 
\Omega_\eta(t)\times(0,\infty)\,,
\\
\diver \bv &= 0\,,\quad \Omega_\eta(t)\times(0,\infty)\,,
\end{align*} where $\bv$ is the fluid velocity, $p$ is pressure, $\rho_f$ is the fluid density, $\bs f$
is an external force and $\sigma_f(\bv,p) = 2\mu \D(\bv) - pI_2$ is the Cauchy stress tensor with 
$\mu$ denoting the fluid viscosity. Geometry of the channel, $\Omega_\eta(t)$ 
is assumed to be the subgraph of an unknown function $\eta$, 
which describes the displacement of the channel's top wall (see Figure \ref{fig:domain_eta}).
The dynamics of $\eta$ is assumed to be governed by a viscoelastic beam type model \cite{Rus92}
\begin{align*}
\varrho_sb\pa_{tt}\eta - D \partial^2_{x}\pa_t\eta + B\partial^4_{x}\eta  
&=-J^{\eta}(x,t) \big(\sigma_f(\bv,p){\bf n}^{\eta}\big )(x,\eta(x,t),t)\cdot{\bf e}_z\,, 
\quad \omega\times(0,\infty)\,.
\end{align*}
Here $\varrho_s$ denotes the structure density, $b$ is thickness of the top wall, 
$D$ is the viscosity coefficient and $B$ is the bending term given by
$B = Eb^3/(12(1 - \nu^2))$, where $E$ is the Young's modulus and $\nu$ the Poisson ratio
of the viscoelastic material. Coupling of the two subsystems (fluid and structure) 
is further strengthened by the continuity condition for normal velocities across the
top boundary: 
\begin{equation*}
\bv(x,\eta(x,t),t) =(0,\partial_t\eta(x,t))\,,\quad \omega\times(0,\infty)\,.
\end{equation*}

We non-dimensionalize the above equations in a standard way. Geometry of the channel is 
non-dimensionalized by the channel length $L$ and other non-dimensional quantities, denoted by hats, 
are introduced as follows:
\begin{align*}
\hat{\bx} = \frac{1}{L}\bx\,,\quad \hat t = \frac{t}{T}\,, 
\quad \hat \bv = \frac{1}{V}\bv\,,\quad \hat{p} = \frac{p}{P}\,,\quad \hat\bef = \frac{1}{F}\bef\,,\quad
\hat\eta = \frac{\eta}{L}\,.
\end{align*}
Setting the time, pressure and force scales as: 
\begin{equation*}
T = \frac{L}{V}\,,\quad P = \frac{\mu V}{L}\,,\quad F = \frac{P}{L}
\end{equation*}
leads to the non-dimensionalized NS equations:
\begin{align}
\Rey\left(\pa_{\hat t}\hat\bv+(\hat\bv\cdot \hat\nabla)\hat\bv\right) - \hat \diver\,\hat\sigma_f(\hat\bv,\hat p) &= \hat\bef\,,\quad 
\Omega_{\hat\eta}(\hat t)\times(0,\infty)\,,\label{A.eq:stokes}\\
\hat \diver\, \hat\bv &= 0\,,\quad \Omega_{\hat\eta}(\hat t)\times(0,\infty)\,,\label{A.eq:divfree}
\end{align}
where $\Rey = \varrho_fVL/\mu$ is the Reynolds number and 
$\hat \sigma_f(\hat\bv,\hat p) = 2\hat\D(\hat\bv) - \hat pI_2$. Assuming that
$\Rey \sim O(1)$ in our system, equations (\ref{A.eq:stokes})--(\ref{A.eq:divfree}) 
correspond to initial equations (\ref{1.eq:stokes})--(\ref{1.eq:divfree}).
\begin{remark}
A customary approach in lubrication theory would be to take different domain scales, 
which would then lead to another rescaled version of the NS system. On the contrary, we avoid 
this a priori scale separation in the fluid domain and follow an approach that is 
ansatz-free and based on carefull examination of energy estimates.
\end{remark}

Similarly, the structure equation turns into 
\begin{equation}
\frac{\varrho_sbV}{\mu}\pa_{{\hat t}{\hat t}}\hat \eta - \frac{D}{\mu L} \partial^2_{\hat x}\pa_{\hat t}\hat \eta 
+ \frac{B}{\mu V L^2}\partial^4_{\hat x}\hat \eta  
=-J^{\hat\eta} \big(\hat \sigma_f(\hat\bv,\hat p){\bf n}^{\hat \eta}\big )\cdot{\bf e}_{\hat z}\,, 
\quad \hat\omega\times(0,\infty)\,,\label{A.eq:elast}
\end{equation}
where $\hat \omega = (0,1)$. On the left-hand side we identify dimensionless numbers,
which we denote by:
\begin{align}
\uvro = \frac{\varrho_sbV}{\mu}\,,\quad \updelta = \frac{D}{\mu L}\,,\quad \upbeta = \frac{B}{\mu V L^2}\,.
\end{align}
For these numbers we assume the following orders of magnitude in terms of a small parameter
$0 < \eps \ll 1$:
\begin{align}\label{A.S1}
\uvro\sim O(\eps)\,,\quad \updelta\sim O(\eps^{-r})\,,\, \, r\in[1,3] \,, \quad \upbeta\sim O(\eps^{-1})\,.
\end{align} 
These are precisely the scaling assumption (S1). 

Let us now advocate (\ref{A.S1}) from a physical point of view. 
We can interpret $\eps$ as the relative channel thickness, i.e.~$\eps = H/L$, 
where $H$ is the nominal value of the channel height. 
If we had rescaled the vertical variable $z$ as $\hat z = z/H$, 
we would obtain the bending term
\begin{equation*}
\tilde\upbeta=\frac{BH}{\mu V L^3} = \upbeta\eps\,.
\end{equation*}
Balancing the fluid pressure, which is $O(1)$ in (\ref{A.eq:elast}), 
with the bending of the structure leads to the requirement
on $\tilde\upbeta$ to be $O(1)$, which in further implies $\upbeta\sim O(\eps^{-1})$.
Moreover, this scaling assumption defines the physical length scale in FSI systems,
namely, $L = \left(B\eps/(V\mu)\right)^{1/2}$. A microfluidic device with structure made out of a 
polymer called polydimethylsiloxane (PDMS), whose characteristic values fit into our theoretical
framework has been designed and experimentaly analyzed in \cite[Experiment S4]{OYN13}.

Next, we discuss the viscoelastic term, which originates from the Kelvin-Voigt model of
viscoelasticity, as it has been argued in \cite{Rus92}. Dimensional analysis reveals that 
$D \sim \mu_sb$, where $\mu_s$ denotes the structure viscosity. Hence,
\begin{equation*}
\updelta \sim \frac{\mu_sb}{\mu L}\,.
\end{equation*}
Since the ratio $\mu_s/\mu$ is very large, typically of order $10^{11}$ for 
channels made of viscoelastic polymers, and on the other hand $b\ll L$, this makes
the assumption $\updelta\sim O(\eps^{-r})$, for some $r\in[1,3]$, plausible.
For simplicity we take one parameter $r$ which takes into account both mechanical 
(viscosity ratio) and geometrical aspects of the microchannel.
Finally, we consider the inertial term. We can write it down as
\begin{equation*}
\uvro = \frac{\varrho_s}{\varrho_f}\frac{b}{L}\Rey\,,
\end{equation*}
from which the assumption $\uvro\sim O(\eps)$ reads as $\varrho_s b\Rey/(\varrho_fL) \sim O(\eps)$.

\subsection{Proof of Lemma \ref{StreamEstimates}}\label{app:stream_est}

Utilizing the Sobolev embedding $L^\infty(\omega)\hookrightarrow H^1(\omega)$ and 
estimate (\ref{DisplacementEstimate}), for every $t\in(0,T)$ we find
\begin{align*}
\|\pa_x\eta^\eps(t)\|_{L^\infty(\omega)} \leq C \|\pa_x\eta^\eps(t)\|_{H^1(\omega)} \leq C\eps\,,
\end{align*}
where $C>0$ is independent of $\eps$.
According to \cite[Proposition 8]{grandmont2016existence}, there exists $C>0$ 
(independent of $\eps$) such that for all $t\in(0,T)$ and all $(x,z)\in\Omega_\eta(t)$
\begin{align*}
|\nabla\psi^\eps(x,z,t)| &\leq C\left(|\pa_x^2\eta^\eps(x,t)| 
+ \frac{|\pa_x\eta^\eps(x,t)|}{\eta^\eps(x,t)} + \frac{|\pa_x\eta^\eps(x,t)|^2}{\eta^\eps(x,t)} 
\right)\\
&\leq C\left(|\pa_x^2\eta^\eps(x,t)| 
+ \frac{\eps}{\eta^\eps(x,t)}\right)\,,
\end{align*}
which proves (\ref{nabla:stream}). From \cite[Proposition 8]{grandmont2016existence},
for all $t\in(0,T)$ we have
\begin{align*}
\|\pa_x\psi^\eps(t)\|_{L^2(\Omega_{\eta^\eps}(t))} &\leq C\|\eta^\eps(t)\|_{L^\infty(\omega)}^{1/2}
\|\pa_x^2\eta^\eps(t)\|_{L^2(\omega)}\,,\\
\|\pa_y\psi^\eps(t)\|_{L^2(\Omega_{\eta^\eps}(t))} &\leq C\|\eta^\eps(t)\|_{L^\infty(\omega)}^{1/4}
\|\pa_x^2\eta^\eps(t)\|_{L^2(\omega)}^{1/2}\left\|\frac{1}{\eta^\eps(t)}\right\|_{L^1(\omega)}^{1/4}\,.
\end{align*}
Now estimates (\ref{EIEkey}) and (\ref{DisplacementEstimate}) immediately provide
(\ref{dx:stream}) and (\ref{dy:stream}).
Combining estimate (\ref{DisplacementEstimate}) and continuity of the Sobolev 
embedding $L^\infty(\omega)\hookrightarrow H^1(\omega)$ gives
$\|\eta^\eps\|_{L^\infty(0,T;L^\infty(\omega))} \leq C\eps$, while energy estimate
(\ref{EIEkey}) yields
\begin{equation*}
\int_0^T\|\pa_{xt}\eta^\eps(s)\|_{L^2(\omega)}^2\dd s \leq C\eps^{r-1}\quad\text{and}\quad
\|\pa_t \eta^\eps\|_{L^\infty(0,T;L^2(\omega))}^2 \leq C\eps^{-4}\,.
\end{equation*}
Employing these estimates in \cite[Proposition 8, eq.~(127)]{grandmont2016existence} we find
\begin{align*}
\|\pa_t\psi^\eps\|_{L^2(\Omega_{\eta^{\eps}}(t)\times(0,T))}
&\leq C \bigg(\int_0^T\big(\|\eta^\eps(s)\|_{L^\infty(\omega)}\|\pa_{xt}\eta^\eps(s)\|_{L^2(\omega)}^2 \\ 
&\quad\qquad + \|\pa_t \eta^\eps(s)\|_{L^2(\omega)}^2\|\pa_x^3\eta^\eps(s)\|_{L^2(\omega)} \big)\dd s \bigg)^{1/2}\\
&\leq C \left( \eps^{r} + \eps^{-4}\int_0^T\|\pa_x^3\eta^\eps(s)\|_{L^2(\omega)}\dd s\right)^{1/2}\,,
\end{align*}
which is (\ref{dt:stream}). Finally, \cite[Proposition 8, eq.~(128)]{grandmont2016existence},
after the time rescaling, gives
\begin{align*}
\|\pa_x^2\psi^\eps\|_{L^2(\Omega_{\eta^{\eps}}(t)\times(0,T))}
&\leq C  \bigg(\int_0^T\big(\|\eta^\eps(s)\|_{L^\infty(\omega)}\|\pa_{x}^3\eta^\eps(s)\|_{L^2(\omega)}^2\\ 
&\quad\qquad + \|\pa_x^2 \eta^\eps(s)\|_{L^2(\omega)}^{3/2}\|\pa_x^3\eta^\eps(s)\|_{L^2(\omega)}^{3/2} \big)\dd s \bigg)^{1/2}\\
&\leq C\left (\int_0^T \left(\eps\|\partial^3_x\eta^{\eps}(s)\|^2_{L^2(\omega)} 
+ \eps^{3/2}\|\partial^3_x\eta^{\eps}(s)\|^{3/2}_{L^2(\omega)}\right) \dd s   \right )^{1/2},
\end{align*}
which finishes the proof.

\bibliographystyle{plain}
%\bibliography{../../../myrefs}

\end{document}